\documentclass[11pt, reqno]{amsart}
	\usepackage{amssymb, latexsym, amsmath, amsfonts, mathrsfs}
	\usepackage{graphicx}
	\usepackage{float}
	\usepackage{tikz}
	\usepackage{enumitem}
	\usepackage[numbers]{natbib}
	\usepackage{hyperref}
	\usepackage{upgreek}
	\makeatletter
\@namedef{subjclassname@2020}{%
  \textup{2020} Mathematics Subject Classification}
\makeatother
	\newtheorem{thm}{Theorem}[section]
	\newtheorem{cor}[thm]{Corollary}
	\newtheorem{lem}[thm]{Lemma}
	\newtheorem{prop}[thm]{Proposition}
	\theoremstyle{definition}
	
	\theoremstyle{remark}
	\newtheorem{rem}[thm]{Remark}
	\numberwithin{equation}{section}
	\newtheorem{ex}[thm]{Example}
	\numberwithin{equation}{section}	
	
	\newcommand{\C}{\mathbb{C}}
	
	\newcommand{\J}{\mathbf{J}}
	\newcommand{\F}{\mathcal{F}}
	\newcommand{\Ju}{\mathcal{J}}
	\newcommand{\G}{\mathscr{G}}
	\newcommand{\g}{\widetilde{G}}
	\newcommand{\V}{\mathcal{V}}
	\newcommand{\h}{\mathsf{H}}
	\newcommand{\M}{\mathsf{M}}
	\newcommand{\Z}{\overline{\mathbf{z}}}
	
	\newcommand{\K}{\mathbf{K}}
	\newcommand{\ko}{\mathcal{K}}
	\newcommand{\Gr}{\mathsf{G}}
	\newcommand{\uo}{\mathcal{U}}
	\newcommand{\U}{\mathbf{U}}
	\newcommand{\N}{\mathcal{N}}
	\newcommand{\s}{\mathcal{S}}

	\newcommand{\du}{\mathbf{d}}
	
	\newcommand{\ov}{\overline}
	\newcommand{\sm}{\setminus}
	\newcommand{\ep}{\epsilon}
	\newcommand{\no}{\noindent}

	\newcommand{\norm}[1]{\|#1\|}
	\newcommand{\abs}[1]{|#1|}
	
	\newcommand{\Aut}{{\rm Aut}(\mathbb C^2)}
	
	\newcommand{\seq}[1]{\{#1_k\}}
	\newcommand{\I}[1]{\textsf{int}(#1)}
	\newcommand{\vo}[1]{\textsf{vol}(#1)}
	\newcommand{\ku}{\mathbf{k}}
	\newcommand{\gr}{\Gr_{\seq{h}}^\pm}
	
\begin{document}
	\title[Semigroup of families of H\'{e}non maps]{Dynamics of semigroups of H\'{e}non maps}
	\author{Sayani Bera} 
\address{Sayani Bera: School of Mathematical and Computational Sciences, Indian Association for the Cultivation of Science, Kolkata-700032, India}
\email{sayanibera2016@gmail.com, mcssb2@iacs.res.in}
\begin{abstract}
The goal of this article is two fold. Firstly, we explore the dynamics of a semigroup of polynomial automorphisms of $\C^2$, generated by a finite collection of H\'{e}non maps. In particular, we construct the positive and negative dynamical Green's functions $G_\G^\pm$ and the corresponding dynamical Green's currents $\mu_\G^\pm$ for a semigroup $\s$, generated by a collection $\G.$ Using them, we show that the positive (or the negative) Julia set of the semigroup $\s$, i.e., $\Ju_\s^+$ (or $\Ju_\s^-$) is equal to the closure of the union of individual positive (or negative) Julia sets of the maps, in the semigroup $\s$. Furthermore, we prove that $\mu_\G^+$ is supported on the whole of $\Ju_\s^+$ and is also the unique positive closed $(1,1)$-current supported on $\Ju_\s^+$, satisfying a semi-invariance relation that depends on the generating set $\G$.

\smallskip Secondly, we study the dynamics of a non-autonomous sequence of H\'{e}non maps, say $\seq{h}$, contained in the semigroup $\s$. Similarly, as above, here too, we construct the non-autonomous dynamical positive and negative Green's function and the corresponding dynamical Green's currents. Further, we use the properties of Green's function to conclude that the non-autonomous attracting basin of any such sequence $\seq{h}$, sharing a common attracting fixed point, is biholomorphic to $\C^2.$
\end{abstract}
\maketitle
\vspace{-.5cm}
\section{Introduction}\label{s:introduction}
In this article, we study the dynamics of a semigroup of polynomial automorphisms of $\mathbb{C}^2$, generated by finitely many \textit{H\'{e}non maps.}  To explain the setup, let $\G$ be a given finite collection of automorphisms of $\C^2$. We will consider the semigroup $\s$ generated by the elements of $\G$ under the composition operation, which will be denoted as
\begin{align}\label{e:S}
	\s=\langle \G \rangle \text{ where } \G=\{{\h}_i: 1 \le i \le n_0\},
\end{align}  
and $n_0 \ge 1$, a positive integer. Furthermore, we assume the maps $\h_i$ are H\'{e}non maps, i.e., for every $1 \le i \le n_0$, there exists $m_i \ge 1$ such that
\begin{align}\label{e:fghm}
	\h_i=H_1^i \circ H_2^i \circ \cdots \circ H_{m_i}^i
\end{align} 
and $H_j^i(x,y)$ is a map of the form
\begin{align}\label{e:ghm}
	H(x,y)=(y,p(y)-ax),
\end{align}
where $a \neq 0$ and $p$ is a polynomial of degree at least 2, for every $1 \le i \le n_0$ and $1 \le j \le m_i$. 

\medskip 
Recall from \cite{UedaBook}, a map of the above form (\ref{e:ghm}) was termed as generalised H\'{e}non map, and as H\'{e}non if $p(y)=y^2+c$, classically. However, the maps of form (\ref{e:fghm}) are more general than (\ref{e:ghm}) and the methods to study maps of form (\ref{e:ghm}), mostly generlises for the entire class. Hence by the terminology \textit{H\'{e}non map} we will mean maps of the form (\ref{e:fghm}).  
Our interest to study the dynamics of semigroups generated by H\'{e}non maps, is for the following facts
\begin{itemize}[leftmargin=14pt]
\item Firstly, a classical result of Friedland--Milnor \cite{FM} states that these maps are essentially the only class of polynomial automorphisms of $\mathbb{C}^2$, exhibiting interesting (iterative) dynamics and have been studied intensively. For instance, see  \cite{HOv:I}, \cite{BS2}, \cite{Dinh-Sibony} etc. Further, their dynamics are known to be connected to the dynamics of polynomial maps in $\C$ (see \cite{BS3}) and they also extend as (bi)rational maps of $\mathbb{P}^2.$
\smallskip
\item Secondly, the study of dynamics of arbitrary families, i.e., non-iterative families of holomorphic self-maps (endomorphisms) is important from the point of view of complex function theory. In particular, the (non-autonomous) basins of attraction\,---\,see Theorem \ref{t:BC} below for definition\,---\,of a sequence of automorphisms of $\C^k$, $k \ge 2,$ with a common attracting fixed point has lead to the construction of pathological domains in $\C^k$ (see \cite{RRpaper}, \cite{Survey}, \cite{F:short}).

\smallskip
\item Also, it is conjectured\,---\,follows as a consequence of a conjecture, originally due to Bedford (\cite{B:openproblem},\cite {FSpaper})\,---\,that a non-autonomous basin of attraction of sequences of automorphisms of $\C^2$, that vary within an infinite (or  finite) collection sharing a common uniformly attracting fixed point should be biholomorphic to $\C^2$. The same is true for autonomous (or iterative) basins of endomorphisms with an attracting fixed point (see \cite[Theorem 9.1]{RRpaper}).
\end{itemize}
To mention here in Section \ref{s:7}, we answer the above for a finite collection of \emph{H\'{e}non maps} by the methods developed in this article to study the semigroup $\s$. It is stated as
\begin{thm}\label{t:BC}
Let $\s$ be as in (\ref{e:S}), such that the generators, $\h_i$, $1 \le i \le n_0$, are attracting on a neighbourhood of origin, i.e., there exist $r>0$ and $0<\alpha<1$ such that 
\begin{align}\label{e:uniformly attracting}
 \|\h_i(z) \|\le \alpha \|z\|\text{ for every }z \in B(0;r).
\end{align}
Then the non-autonomous basin of attraction at the origin of every sequence $\seq{h} \subset \s$ defined as $\Omega_{\seq{h}}:=\{z \in \C^2: h_k \circ h_{k-1}\circ \cdots \circ h_1(z) \to 0 \text{ as }k\to \infty\}$ is biholomorphic to $\C^2$.
\end{thm}

Further, we study a few particular cases of an infinite collection or parametrised families of H\'{e}non maps sharing a common attracting fixed point with `uniform bounds'. In particular, the following example is obtained as a consequence of Corollary \ref{c:BC} (also see Example \ref{e:skew}), in comparison to Theorem 1.4 and 1.10 in \cite{F:short}.
\begin{ex}
Let $H_k(x,y)=(a_k y, a_k x+y^2)$ where $p$ is a polynomial of degree at least $2$ and $c<|a_k|<d$ for every $k \ge 1$, with $0<c<d<1$. Then the 	basin of attraction of the sequence $\seq{H}$, i.e., $\Omega_{\seq{H}}$ (as defined in Theorem \ref{t:BC}) is biholomorphic to $\C^2$. 
\end{ex}

\begin{itemize}[leftmargin=14pt]
\item Lastly, the study of dynamics of rational semigroups on $\mathbb{P}^1$ is an interesting and widely studied area. This setup was introduced by Hinkkanen--Martin, in \cite{HM}, motivated by their connection to the dynamics of Kleinian groups on the Riemann sphere, observed in \cite{GM:Kleiniangroup}.
\end{itemize}

Our primary goal in this article is, to explore the dynamics of a semigroup of H\'{e}non maps both in $\mathbb{P}^2$ and $\C^2$, motivated by the study of dynamics of rational semigroups in $\mathbb{P}^1.$ In particular, we will attempt to connect between results from iterative dynamics of H\'{e}non maps of the form (\ref{e:fghm}) and semigroup dynamics of rational maps in $\mathbb{P}^1$ to the current setup. Later, we will generalise a few results appropriately in the setup of non-autonomous families to obtain the aforementioned applications.

\smallskip Let $X$ be a complex manifold and $\s$ be an \textit{infinite} family of holomorphic self-maps of $X.$ The \textit{Fatou set} for the family $\s$ is the largest open set of $X$ where the family $\s$ is normal, i.e.,
\[\mathcal{F}_{\s}=\{z \in X: \text{there exists a neighbourhood of } z  \text{ where the family } \s \text{ is normal}\}.\]
The \textit{Julia set} $\mathcal{J}_{\s}$, is the complement of the Fatou set in $X.$

\smallskip As reported earlier, the setup considering $X=\mathbb{P}^1$ and $\s$, a semigroup generated by more than one rational map of degree at least 2, was introduced in \cite{HM} and later on has been explored extensively. A major difficulty in this framework\,---\,as compared to the iterative dynamics\,---\,is neither the Julia set nor the Fatou set is completely invariant, in general.

\smallskip
It is a classical result of Brolin \cite{Brolin} that says - if $\s$ is the semigroup of iterates of a (single) polynomial map $p$ of degree at least 2,  the limiting distribution of points in the preimages of a generic point $z \in \mathbb{P}^1$, corresponds to the equilibrium measure of the Julia set. Further, the potential associated with this measure, i.e., the Green's function of the Julia set can be constructed via the dynamics of $p$. The equidistribution of the iterated preimages of a generic point $z \in \mathbb{P}^1$, i.e., the limiting distribution is independent of the (generic) $z$, was established for the iterations of rational map of $\mathbb{P}^1$ by Lyubich in \cite{Lyubich}. Boyd in \cite{Boyd}, extended Lyubich's method and constructed an equidistributed measure supported on the Julia set of a finitely generated semigroup of rational maps (of degree at least 2) in $\mathbb{P}^1$. 
For a finitely generated semigroup of polynomials of degree at least 2, Boyd's measure is not, in general, the equilibrium measure of its Julia set. Recently, in \cite{L:semigroups} the latter measure is interpreted as an equilibrium measure in the presence of an external field, which is given by a generalisation of the Greens function\,---\,attributed as `dynamical Greens function'.

\smallskip To note, equidistributed measures exist for dynamics of certain meromorphic correspondences on compact connected K\"{a}hler manifolds, of appropriate intermediate degree (see \cite{DS:correspondences}). However,  birational maps of $\mathbb{P}^2$ obtained from extension of H\'{e}non maps do not belong to the above category. Also, for iterative families of a H\'{e}non map the Julia set is captured via the support of a unique positive closed $(1,1)$-current of mass 1, obtained by the action of $dd^c$-operator on the pluri-complex Green's function of the Julia set.  Furthermore, it is an equidistributed current in $\C^2$, in the sense, that it can be recovered as a limit of appropriately weighted preimages of an algebraic variety in $\C^2$\,---\,see \cite[Theorem 4.7]{BS2} or \cite[Corollary 6.7]{Dinh-Sibony}.
 
 \smallskip  To mention here, construction of currents for non-autonomous families of H\'{e}non maps have been done on an appropriate bounded region containing the origin, in \cite{DS:horizontal}, via the fact they are {\it horizontal}. Here, we construct a (similar) global equidistributed dynamical Green's current on $\mathbb{P}^2$ in Corollaries \ref{t:result 4} and  \ref{t:NA currents}\,---\,using the dynamical Green's functions, constructed by generalising ideas from \cite{Boyd}, \cite{DS:horizontal}, \cite{L:semigroups}\,---\,both for the semigroups $\s$ and non-autonomous families. Thus, obtaining the uniqueness of the currents upto a semi-invariance property for the semigroups $\s$, stated in Corollary \ref{c:result 5}. Also, see Remark \ref{r:analogy}, for details.
 
\smallskip Let us first recall a few important properties of iterations of a H\'{e}non map $\h$. The pluri-complex Green's function (see \cite{KlimekBook} for the definition) associated to the Julia set of iterates of $\h$ or $\h^{-1}$, say $G_{\h}^+$ or $G_\h^-$ respectively, can be recovered via the dynamics. In particular, if $d_\h$ is the degree of the map $\h$ then
\[ G_\h^\pm(z)=\lim_{k \to \infty} \frac{\log^+ \|\h^{\pm k}(z)\|}{d_\h^k} \text{ and } \mu_\h^\pm(z)=\frac{1}{2\pi}dd^c (G_\h^\pm),\]
where $\log^+ x=\max\{\log x, 0\}$ for every $x>0$ and $\norm{\cdot}$ be the supremum norm in $\C^2$. Also, $\mu_\h^\pm$ are the closed positive $(1,1)$-(equidistributed) currents as mentioned previously.  Note that the above definition holds for any norm on $\C^2$, however for the sake of convenience we will use the notation $\norm{\cdot}$ to denote the supremum norm, throughout this article. 

\smallskip Now, let $\s$ be the semigroup as introduced in (\ref{e:S}), i.e., $\s=\langle\G\rangle$ where $\G=\{\h_i:1 \le i \le n_0\}$ and $\h_i$ are H\'{e}non maps of the form (\ref{e:fghm}) and of degree $d_i \ge 2$ for every $1 \le i \le n.$ We first generalise a few definitions and observe some basic results regarding the semigroup $\s$ in Section \ref{s:2}. Particularly, we note that $\s$ might have more than one generating set, however it has a unique minimal generating set. 

\smallskip In Section \ref{s:3}, we generalise the construction of positive and negative Green's functions, i.e., the functions $G_\h^\pm$ noted above, in the setup of the  semigroup $\s.$ To do the same, we define the total degree of the semigroup $\s$ with respect to the generating set $\G$ as $D_{\G}=\sum_{i=1}^{n_0} d_i$, and consider the sequence of plurisubharmonic functions $G_k^\pm$ on $\mathbb{C}^2$ defined as
\begin{align}\label{e:Green sequence}
G_k^+(z)=\frac{1}{D_{\G}^k} \sum_{h \in \G_k} \log^+\|h(z)\| \text{ and }G_k^-(z)=\frac{1}{D_{\G}^k} \sum_{h \in \G_k} \log^+\|h^{-1}(z)\|,
\end{align}
where $\G_k$ denote the elements of the semigroup $\s$ of length $k$ with respect to the generating set $\G$, i.e.,
$ \G_k=\{\h_{i_1} \circ \cdots\circ \h_{i_k}: 1 \le i_j \le n_0 \text{ and } 1 \le j \le k\}.$  We prove that the pointwise limits of the sequences $\{G_k^\pm\}$ constructed in (\ref{e:Green sequence}) exist, which is stated as
\begin{thm}\label{t:result 1}
The sequences $\{G_k^\pm\}$ converge pointwise to plurisubharmonic, continuous functions $G_{\G}^\pm$ on $\mathbb{C}^2$, respectively.
\end{thm}
Henceforth, the functions $G_{\G}^\pm$ will be referred as the \textit{dynamical positive (or negative) Green's function} associated to the semigroup $\s$ generated by the set $\G=\{{\h}_i: 1 \le i \le n_0\}$. The need to specify the generating set $\G$ is important as the semigroup $\s$ may admit multiple generating sets. Also, note that the functions $G_\G^\pm$ satisfy the following semi-invariance relation right by the construction (\ref{e:Green sequence}) and Theorem \ref{t:result 1}. 
\begin{cor}\label{c:result 2}
$ \sum_{i=1}^{n_0}G_{\G}^+ \circ \h_i(z)=D_{\G}. G_{\G}^+(z)$ and $ \sum_{i=1}^{n_0}G_{\G}^- \circ \h_i^{-1}(z)=D_{\G} .G_{\G}^-(z).$
\end{cor}
 Thus, as consequences of the proof of Theorem \ref{t:result 1}, we note that the functions $G_\G^\pm$ admit logarithmic growth on appropriate regions, and the \textit{strong filled} positive and negative Julia sets of the semigroup $\s$ are pseudoconcave sets (see Corollary \ref{c:non-zero on uo+} and Remark \ref{r:ko pseudoconcave}). 

\smallskip
Next, we analyse the Julia sets $\Ju_\s^\pm$ and the properties the \textit{dynamical Green's $(1,1)$-currents} associated to the functions $G_\G^\pm$, defined as $\mu_\G^\pm=\frac{1}{2 \pi} dd^c G_\G^\pm$. Consequently, in Section \ref{s:4}, we prove the analogue to Corollary 2.1 from \cite{HM}\,---\,an important fact from the dynamics of semigroups of the rational maps on $\mathbb{P}^1$\,---\,via the supports of $\mu_\G^\pm.$
\begin{thm}\label{t:result 3}
The positive and negative Julia sets corresponding to the dynamics of the semigroup $\s$ is equal to the closure of the union of the (positive and negative) Julia sets of the elements of $\s$ respectively, i.e.,
$\displaystyle\Ju_\s^+=\overline{\bigcup_{h \in \s} J_h^+} \text{ and } \Ju_\s^-=\overline{\bigcup_{h \in \s} J_h^-}.$

\noindent
Further, the positive and the negative dynamical Green's currents $\mu_\G^\pm$ are $(1,1)$-closed positive currents of mass 1 supported (respectively) on the Julia sets, i.e.,
$ \text{Supp }(\mu_\G^\pm) = \Ju_\s^\pm.$
\end{thm}
Thus from Theorem \ref{t:result 1} and the above, $G_\G^\pm$ is actually pluriharmonic on the Fatou sets $\mathcal{F}_\s^\pm$. In Section \ref{s:5}, we study the extension of the currents $\mu_\G^\pm$ to $\mathbb{P}^2$ and prove that they are limits of (weighted) equidistributed projective varieties, in the spirit of \cite[Theorem 6.2]{Dinh-Sibony}. Also, consequently we observe the following uniqueness of $\mu_\G^+$ upto a semi-invariance property.
\begin{cor}\label{c:result 5}
The current $\mu_\G^+$ is the unique current of mass 1 supported on $\Ju_\s^+$ and the current $\mu_\G^-$ is the unique current of mass 1 supported on $\Ju_\s^-$ satisfying the following semi-invariance relations (respectively)
\begin{align}\label{e:functorial current}
	\frac{1}{D_\G} \sum_{i=1}^{n_0} \h_i^*(\mu_\G^+)=\mu_\G^+ \text{ and }	\frac{1}{D_\G} \sum_{i=1}^{n_0} {\h_i}_*(\mu_\G^-)=\mu_\G^-.
\end{align}
\end{cor}
In Section \ref{s:6}, we consider the dynamics of a non-autonomous sequence of H\'{e}non maps, say $\seq{h} \in \s$ and prove that there exist plurisubharmonic and continuous dynamical (positive and negative) Green's functions, denoted by $\Gr_{\seq{h}}^\pm$, with logarithmic growth. Thus
$\mu_{\seq{h}}^\pm=\frac{1}{2\pi}dd^c(\Gr_{\seq{h}}^\pm)$,
are positive $(1,1)$-currents of mass 1, supported on the positive and negative Julia sets of the sequence $\seq{h}.$ Also, we obtain the analogs to the equidistribution results, i.e., Corollaries \ref{t:result 4} and \ref{c:result 4}, in this setup of non-autonomous dynamics. However, we will discuss them  briefly as most of the ideas are similar to that realised in Section \ref{s:5} and, depends upon the existence of the dynamical Green's function with suitable growth at infinity.

\smallskip
Finally, we study the non-autonomous attracting basin of a sequence of H\'{e}non maps of the form (\ref{e:fghm}), admitting a uniformly attracting behaviour, on a neighbourhood of the origin. Further, we prove Theorem \ref{t:BC}, as an application of the existence of Green's functions $\Gr^\pm_{\seq{h}}$ and enlist a few more applications, which follows from the technique. All of these affirmatively answers a few particular cases of the equivalent formulation of the Bedford Conjecture, in $\C^2$ for H\'{e}non maps\,---\,as alluded to in the beginning.
\subsection*{Acknowledgement} The author would like to thank the anonymous referee for carefully reading the manuscript and suggesting helpful comments. 

\section{Some basic definitions and preliminaries}\label{s:2}
In this section, we first observe a proposition about the generating set $\G$ of the semigroup $\s$ as in (\ref{e:S}), which might not be unique, always. Recall the setup from Section \ref{s:introduction}, let $\displaystyle \G=\{\h_i: 1 \le i \le n_0\}$ where $\h_i$'s are H\'{e}non maps of the form (\ref{e:fghm}), with degree $d_i \ge 2$. The \textit{total degree} of the semigroup $\s$ with respect to the generating set $\G$ is  $D_\G=\sum_{i=1}^{n_0}d_i$ and $\G_k$ is the set of all elements of length $k$, $k \ge 1$ in the semigroup $\s$ with respect to $\G$. 
\begin{prop}\label{p:minimal generator}
Let $\s$ be a finitely generated semigroup as in (\ref{e:S}), then there exists a unique minimal set $\G_0$ of maps of the form (\ref{e:fghm}) that generates $\s$, i.e., any set of generators $\G$ of $\s$ is a superset of $\G_0.$
\end{prop}
\begin{proof}
For $n \ge 1$, let $\s(n)=\{H \in \s: \text{degree of }H \text{ is }n\}.$ Note that  $\s(1)$ is an empty set, as the degree of every element in the generating set $\G$ is at least $2.$ However, $\s(n)$ for every $n \ge 1$, need not necessarily be empty but is always a finite set. We will construct the minimal generating set $\G_0$ inductively, such that it terminates after finitely many steps. Let 
\[
 A_2:=\s(2), A_3:=\s(3) \setminus \big \langle A_2\big \rangle, A_4:=\s(4) \setminus \big \langle A_2 \cup A_3\big \rangle, \hdots, A_n:=\s(n) \setminus \big \langle \bigcup_{i=2}^{n-1}A_i \big\rangle.
\]
Since $\s$ is finitely generated, there exists an $n_0 \ge 1$, such that $A_n=\emptyset$ for $n > n_0$ and $A_{n_0} \neq \emptyset.$ Let $$\G_0=\bigcup_{i=2}^{n_0}A_i.$$
Note that by construction, any element in $\G_0$ is not generated by lower degree maps of form (\ref{e:fghm}). Further as $A_n=\emptyset$ for every $n>n_0$, $\G_0$ is the minimal set generating $\s.$
\end{proof}
\begin{rem}
Thus the \textit{total degree} of a semigroup $\s$ is dependent on the generating set $\G$ and is not unique, in general. Consequently, the sequence of plurisubharmonic functions $\{G_k^\pm\}$ defined in (\ref{e:Green sequence}) and the  \textit{positive} and \textit{negative dynamical Green's function} is also dependent on the generating set $\G$ of the semigroup $\s.$
\end{rem}
Next, we revisit and introduce a few important definitions (and notations) with respect to the dynamics of the semigroup $\s$, that are independent of the generating set $\G$. 
\begin{itemize}[leftmargin=14pt]
\item Let $\s^-$ denote the semigroup of maps comprising of the inverse of the maps that belong to $\s$ and $\G^-$ the inverse of the elements that belong to $\G$, i.e., 
\[ \G^-=\{\h_i^{-1}: 1 \le i \le n_0\} \text { and } \s^-=\langle \G^-\rangle.\]
\item  The Fatou sets of $\s$ and $\s^-$)\,---\,as stated in Section \ref{s:introduction}\,---\,is denoted by $\F_\s^+$ and $\F_\s^-$ respectively. The\textit{ positive and negative Julia sets} are denoted by $\Ju_\s^\pm.$ 

\smallskip
\item We consider the following {\it two} alternatives for the filled positive and negative Julia sets.

\smallskip
\begin{enumerate}
\item The \textbf{strong} positive (or negative) filled Julia set is defined as the collection of all the points $z \in \mathbb{C}^2$ such that for every sequence $\{h_k\} \subset \s$, the sequence $\{h_k(z)\}$ (or the sequence $\{h_k^{-1}(z)\}$, respectively) is bounded, i.e.,

\smallskip
$\displaystyle {\ko}^+_{\s}=\big\{z \in \C^2: \text{for every sequence }\{h_k\}\subset \s \text{ the sequence } \{h_k(z)\} \text{ is bounded}\big\},$

\smallskip 
$\displaystyle {\ko}^-_\s=\big\{z \in \C^2: \text{for every sequence }\{h_k\}\subset \s \text{ the sequence } \{h_k^{-1}(z)\} \text{ is bounded}\big\}.$

\smallskip\item The \textbf{weak} positive filled Julia set is defined as the collection of all the points $z \in \mathbb{C}^2$ such that there exists a sequence $\{h_k\} \subset \s$ with $h_k \in \G_{n_k}$, where $n_k \to \infty$ as $k \to \infty$, and $\{h_k(z)\}$ is bounded. Similarly we define the \textbf{weak} negative filled Julia set, i.e.,

\smallskip
$\displaystyle \K^+_\s=\big\{z \in \C^2: \text{ there exist } h_k \in \G_{n_k} \text{ such that } n_k \to \infty \text{ and }\{h_k(z)\} \text{ is bounded}\big\},$

\smallskip
$\displaystyle \K^-_\s=\big\{z \in \C^2: \text{ there exist } \tilde{h}_k \in \G_{n_k} \text{ such that } n_k \to \infty \text{ and }\{\tilde{h}_k^{-1}(z)\} \text{ is bounded}\big\}.$
\end{enumerate}

\smallskip \no Note that $\ko_\s^\pm \subset \K_\s^\pm$ and these sets are uniquely associated to the semigroup $\s$.

\smallskip
\item Similarly as above we introduce the \textbf{weak} and \textbf{strong} escaping sets $\uo^\pm_\s$ and  $\U_\s^\pm$
\[ \uo_\s^\pm=\C^2 \setminus \ko_\s^\pm \text{ and } \U_\s^\pm=\C^2 \setminus \K_\s^\pm.\]
Note that $\U_\s^+$ is the Fatou component at infinity with respect to the dynamics of the semigroup $\s.$ Similarly $\U_\s^-$ is the Fatou component at infinity for $\s^-$. 

\smallskip

\item Finally, we define the \textit{cumulative positive and negative Julia sets} for the semigroup $\s$, i.e.,
\[ \J_\s^+=\overline{\bigcup_{h \in \s} \Ju_h^+} ,\, \J_\s^-=\overline{\bigcup_{h \in \s} \Ju_h^-}.\] 
\end{itemize}
Observe that, either of the sets $\ko_\s^+$ or $\ko_\s^-$ or both may be empty for some  semigroups $\s$, of form (\ref{e:S}). However, this situation does not affect the dynamics as such.

\smallskip
We will now explore some important properties of the sets introduced above via the filtration properties of the elements of $\s$, on appropriate domains. To discuss this in detail, let us first recall the definition of the sets $V_R$ and $V_R^\pm$ for some $R>0$, introduced in \cite{HOv:I} (or \cite{BS2}) for filtering the dynamics of (finite) compositions of generalised H\'{e}non maps. They are $V_R=\big\{(x,y) \in \C^2: \max \{|x|,|y|\} \le R\big\}$, the polydisk of radius $R$ and
\begin{align*} 
V_R^+=\big\{(x,y) \in \C^2: |y| \ge \max\{|x|,R\}\big\}, \;\;
V_R^-=\big\{(x,y) \in \C^2: |x| \ge \max\{|y|,R\}\big\}.
\end{align*}
Also recall the subsets $\G_k$ of $\s =\langle \G \rangle$, for every $k \ge 1$, defined as
\[\G_k=\{\h_{i_1} \circ \cdots\circ \h_{i_k}: 1 \le i_j \le n_0 \text{ and } 1 \le j \le k\}, \text{ where } \G_1=\G=\{\h_i:1 \le i \le n_0\}.\] 
We first record the dynamical behaviour of the semigroup $\s$ on $V_R^\pm$ for an appropriate $R>0$.
\begin{lem}\label{l:filtration}
There exists $R_\s>0$ such that for every $R>R_\s$, 
\[\overline{h(V_R^+)} \subset V_R^+ \text{ and } \overline{h^{-1}(V_R^-}) \subset V_R^-,
\text{ i.e., }
V_R \cap h(V_R^+)=\emptyset \text{ and }V_R \cap h^{-1}(V_R^-)=\emptyset\] 
whenever $h \in \s.$ Further let $\{h_k\} \subset \s$ such that $h_k \in \G_k$ for every $k \ge 1$, then there exists a sequence positive real numbers $R_k \to \infty$ satisfying
$$V_{R_k} \cap h_k(V_R^+)=\emptyset \text{ and }V_{R_k} \cap h_k^{-1}(V_R^-)=\emptyset.$$
\end{lem}
\begin{proof}
	Recall from \cite{BS2}, for $R>0$ (sufficiently large) there exists $0<m<M$ such that
	\begin{align*}
	m|y|^{d_i} <|\pi_2 \circ \h_i(x,y)|< M|y|^{d_i} \text{ on }V_R^+, \;
	m|x|^{d_i} < |\pi_1 \circ \h_i^{-1}(x,y)|< M|x|^{d_i} \text{ on }V_R^-.
	\end{align*}
Recall from (\ref{e:fghm}) 
\[\h_i=H_1^i \circ H_2^i \circ \cdots \circ H_{m_i}^i,\] where $H_j^i(x,y)=(y,p_j(y)-a_j x)$. Thus degree of $\pi_1 \circ \h_i < \text{ the degree of }\pi_2 \circ \h_i$, for every $1 \le i \le n_0$. Now as $V_R^+$ is a closed subset of $\C^2$ by definition, the above identity further implies that $|\pi_1 \circ \h_i(z)|<|\pi_2 \circ \h_i(z)|$ and $|\pi_2 \circ \h_i(z)|>R$, for $z \in V_R^+$, $R>0$, sufficiently large. In particular, there exists $R>0$, large enough such that
\begin{align}\label{e:VR+}
 \ov{\h_i(V_R^+)}=\h_i(V_R^+) \subset \I{V_R^+},\; \ov{\h_i^{-1}(V_R^-)}=\h_i^{-1}(V_R^-) \subset \I{V_R^-}\text{ for every }1 \le i \le n_0. 
\end{align}
\no Let $d_0=\min\{d_i: 1 \le i \le n_0\} \ge 2$ and $R_\s>1$ be sufficiently large such that $1<R_\s<m R_\s^{d_0}.$ 
Hence from (\ref{e:VR+}), for $R_k=mR^{d_0}_{k-1}$ whenever $k \ge 2$ and $R_1>R_\s$. Thus the proof.
\end{proof}
The constant $R_\s>0$ obtained in Lemma \ref{l:filtration} is actually independent of the generators and will be referred along, as the \textit{radius of filtration} for the semigroup $\s.$ 
\begin{rem}\label{r:filtration estimate}
Note that in the above proof we may assume $0<m<1<M$, such that for every $R>R_\s$ and $1 \le i \le n_0$
\[m|y|^{d_i} <|\pi_2 \circ \h_i(x,y)|< M|y|^{d_i} \text{ on }V_R^+, \text{ and }
	m|x|^{d_i} < |\pi_1 \circ \h_i^{-1}(x,y)|< M|x|^{d_i} \text{ on }V_R^-.\]
\end{rem}
\begin{prop}\label{p:K+ closed}
The sets $\K_\s^\pm$ and $\ko_\s^\pm$ are closed subsets of $\C^2$ and $\ko_\s^\pm \subset \K_\s^\pm \subset V_R \cup V_R^\mp$ (respectively) for $R\ge R_\s.$
\end{prop}
\begin{proof}
Let $U_0=\textsf{int}(V_R^+)$ and let $\{\U_k\}$, $\{\uo_k\}$ be the sequences of open subsets defined as
\[\U_k=\bigcap_{h \in \G_k} h^{-1}(U_0) \text{ and } \uo_k=\bigcup_{h \in \G_k} h^{-1}(U_0).\]
Then 
\[\ov{\U_k}=\bigcap_{h \in \G_k} h^{-1}(\ov{U_0}) \text{ and } \ov{\uo_k}=\bigcup_{h \in \G_k} h^{-1}(\ov{U_0}).\]
Since by (\ref{e:VR+}), $\h_i (z_0) \in \textsf{int}(V_R^+)$ for every $1 \le i \le n_0$ whenever $z_0 \in V_R^+=\ov{U_0}$, we have $\ov{U_0} \subset \h_i^{-1}(U_0)$. Hence $\ov{U_0} \subset \U_1 \subset \uo_1.$ Further for every $h \in \G_k$, $h^{-1}(\ov{U_0}) \subset h^{-1}\circ \h_i^{-1}(U_0)$ for every $\h_i \in \G$, where $k \ge 1$. Thus 
\begin{align}\label{e:U-inclusion}
\ov{\U_k} \subset \U_{k+1} \text{ and } \ov{\uo_k} \subset \uo_{k+1}.	
\end{align}
Let $$\U^+=\bigcup_{k \ge 0} \U_k \text{ and } {\uo}^+=\bigcup_{k \ge 0} \uo_k.$$

Observe that for every $k \ge 1$, $h(\U_k) \subset V_R^+$ whenever $h \in \G_k$. Hence $\U^+ \subset \U_\s^+.$ Now for $z \in \C^2 \setminus \U^+$, note that $h_k(z) \in V_R \cup V_R^-$ for every sequence $\{h_k\} \subset \s.$ Let $z_0 \in \U_\s^+ \cap (\C^2 \setminus \U^+)$. Then there exists a sequence $\{\tilde{h}_k\} \subset \s$ such that $\tilde{h}_k \in \G_{n_k}$ with $n_k \to \infty$ as $k \to \infty$ and $\tilde{h}_k(z_0) \in V_R^-$ for every $k \ge 1$, i.e., $z_0 \in \tilde{h}_k^{-1}(V_R^-)$. Hence by Lemma \ref{l:filtration}, $z_0 \notin \cup_k V_{R_{n_k}}=\C^2$, which is a contradiction! Thus $\U^+=\U_\s^+$. A similar argument works for $\U^-=\U_\s^-$.

\smallskip
Similarly for $z \in {\uo_k^+}$ there exists $h_k \in \G_k$ such that $h_k(z) \in V_R^+$, hence $\uo^+ \subset \uo_\s^+$. Now for $\tilde{z}_0 \in \uo_\s^+\cap (\C^2 \setminus \uo^+)$, as in the above case, there exists a sequence $\{\tilde{h}_k\}$ such that $\tilde{h}_k \in \G_{n_k}$ with $n_k \to \infty$ as $k \to \infty$  and $\tilde{h}_k(\tilde{z}_0) \in V_R^-$ for every $k \ge 1$, i.e., $\tilde{z}_0 \in \tilde{h}_k^{-1}(V_R^-)$. Hence by Lemma \ref{l:filtration}, $\tilde{z}_0 \notin \cup_k V_{R_{n_k}}=\C^2$, which is a contradiction! Thus $\uo^+=\uo_\s^+$. A similar argument works for $\uo^-=\uo_\s^-$.

\smallskip\no Note that the above observations also proves that $V_R^\pm \subset \U_\s^\pm \subset \uo_\s^\pm$, hence $\ko_\s^\pm \subset \K_\s^\pm \subset V_R \cup V_R^\mp$ (respectively).
\end{proof}
Further, from the proof of Proposition \ref{p:K+ closed} we get
\begin{cor}\label{c:U+ and U-}
The escaping sets of $\s$ can be further realised as
\begin{enumerate}
\item $\displaystyle \U_\s^+=\bigcup_{k \ge 1}\bigcap_{h \in \G_k} h^{-1}(V_R^+) \text{ and }\;\U_\s^-=\bigcup_{k \ge 1}\bigcap_{h \in \G_k} h(V_R^-)$;
\item $\displaystyle \uo_\s^+=\bigcup_{k \ge 1}\bigcup_{h \in \G_k} h^{-1}(V_R^+) \text{ and }\;\uo_\s^-=\bigcup_{k \ge 1}\bigcup_{h \in \G_k} h(V_R^-)$.
\end{enumerate}
\end{cor}
\begin{rem}\label{r:inclusions}
Note that $\K_\s^\pm \setminus \textsf{int}(\ko_\s^\pm)=\overline{\uo_\s^\pm \setminus \U_\s^\pm}$ and the Julia sets $\Ju_\s^\pm \subset \K_\s^\pm \setminus \textsf{int}(\ko_\s^\pm)$, however, they might not be equal. Thus so far we have the straightforward inclusion relation
\[\J_\s^\pm \subset \Ju_\s^\pm \subset \K_\s^\pm \setminus \textsf{int}(\ko_\s^\pm).\]
Also both $\partial \ko_\s^\pm, \partial \K_\s^\pm \subset \Ju_\s^\pm$, i.e., they may be proper subsets $\Ju_\s^\pm$, unlike the iterative dynamics of H\'{e}non maps. So it leads to the question: Is $\Ju_\s^\pm = \K_\s^\pm \setminus \textsf{int}(\ko_\s^\pm)$ or $\Ju_\s^\pm=\partial \ko_\s^\pm \cup \partial \K_\s^\pm$?
\end{rem}
\begin{figure}[H]
 \centering
\begin{minipage}{.5\textwidth}
  \centering
  \includegraphics[scale=.9]{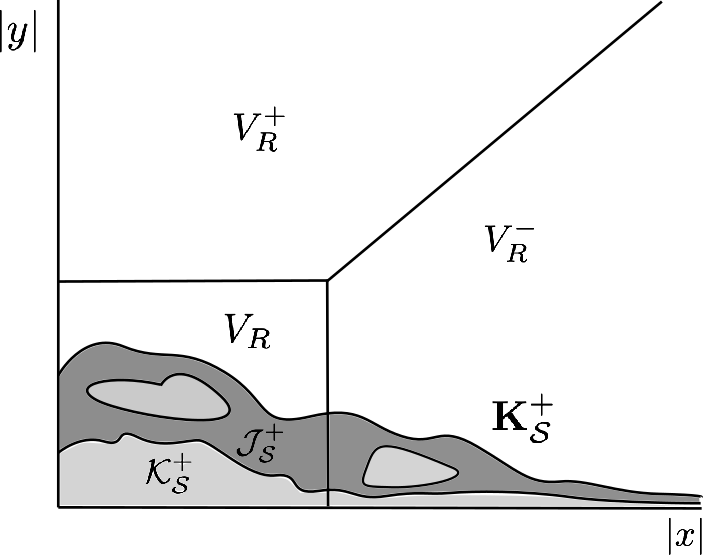}
\end{minipage}%
\begin{minipage}{.5\textwidth}
  \centering
  \includegraphics[scale=.9]{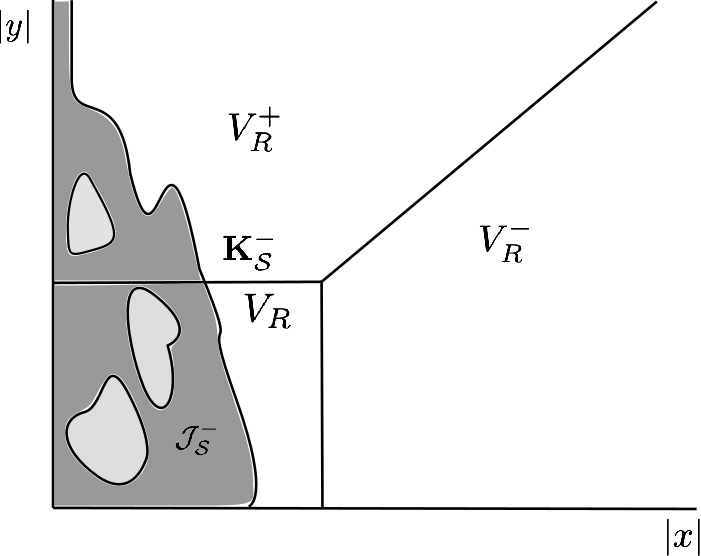}
  \end{minipage}
  \caption{$\ko_\s^+$, $\K_\s^+$ $\Ju_\s^+$, $\K_\s^-$ and $\Ju_\s^-$, with $\ko_\s^-=\emptyset$}
\label{f:filtration}
 \end{figure} 
The shaded region in Figure \ref{f:filtration} corresponds to the weak filled Julia sets $\K_\s^\pm$, with the lighter shade representing the Fatou components and the darker shades representing the Julia sets $\Ju_\s^\pm$ contained in them, respectively. 

\smallskip Finally we conclude this section, by observing the following crucial fact which will be very important for further computations
\begin{lem}\label{l:step 4}
Let $R>R_\s$, where $R_\s>0$ is the radius of filtration for the semigroup $\s$ and $C$ be a compact subset of $\C^2$.
\begin{enumerate}
	\item [(i)] Then there exists a positive integer $k_C \ge 1$ such that $h(C) \subset V_R \cup\I{V_R^+}$ for every $h \in \G_k$ and $k \ge k_C.$
	\item [(ii)] Then there exists a positive integer $\tilde{k}_C \ge 1$ such that $h^{-1}(C) \subset V_R\cup \I{V_R^-}$ for every $h \in \G_k$ and $k \ge \tilde{k}_C.$
\end{enumerate}
\end{lem}
\begin{proof}
 Suppose the statement (i) is not true, i.e., there exist a sequence $\{z_n\} \subset C$ and a sequence $\{h_n\}$ with $h_n \in \G_{k_n}$, $k_n \to \infty$ as $n \to \infty$ such that $h_n(z_n) \in V_R^-$. Then by Lemma \ref{l:filtration}, $z_n \in h_n^{-1}(V_{R}^-) \subset V_{R_{k_n}}^-$. As $R_{k_n} \to \infty$ for $n \to \infty$, $\norm{z_n} \to \infty$. This contradicts the fact that $\{z_n\}$ is contained in the compact set $C$.

\smallskip A similar argument works for part (ii). 
\end{proof}

\section{Proof of Theorem \ref{t:result 1}}\label{s:3}
In this section, we will first complete the proof of Theorem \ref{t:result 1} and consequently observe a few important corollaries. To begin, let us recall the definition of the sequence of plurisubharmonic functions $\{G_k\}$, introduced in Section \ref{s:introduction},
\begin{align*}
G_k^+(z)=\frac{1}{D^k} \sum_{h \in \G_k} \log^+\|h(z)\| \text{ and }G_k^-(z)=\frac{1}{D^k} \sum_{h \in \G_k} \log^+\|h^{-1}(z)\|,
\end{align*}
where $D=D_\G$ is the total degree corresponding to the generating set $\G$ of the semigroup $\s.$
First, we note the following straightforward consequence from the results in Section \ref{s:2}.
\begin{lem}\label{l:compact sets on U+}
	Fix an $R \ge R_\s$, the radius of filtration for the semigroup $\s.$ Then
\begin{itemize}
	\item for a compact set $C^+ \subset \U_\s^+$ there exists a positive integer $\N_{C^+} \ge 1$ such that 
	$h(C^+) \subset V_{R}^+$ whenever $h \in \G_k$, $k \ge \N_{C^+}$;
	\item for a compact set $C^- \subset \U_\s^-$ there exists a positive integer $\N_{C^-} \ge 1$ such that 
	$h^{-1}(C^-) \subset V_{R}^-$ whenever $h \in \G_k$, $k \ge \N_{C^-}.$
\end{itemize}
\end{lem}
\begin{proof}
By Corollary \ref{c:U+ and U-} and the proof of Proposition \ref{p:K+ closed}, there exists $\N_{C^+}, \N_{C^-}\ge 1$ such that 
$$C^+ \subset \bigcap_{h \in \G_k} h^{-1}(V_{R}^+),\; C^- \subset \bigcap_{h \in \G_l} h(V_{R}^-) \text{ whenever } k \ge \N_{C^+} \text{ and } l \ge \N_{C^-}. \qedhere$$
\end{proof}
Now we are ready to present the proof of Theorem \ref{t:result 1}, which involves some steps. 
\begin{proof}[Proof of Theorem \ref{t:result 1}]
Let $R\ge R_\s$ be as in Lemma \ref{l:compact sets on U+}.

\smallskip\no 
\textit{Step 1: }The sequence $\{G_k^\pm\}$ converges uniformly to a pluriharmonic function on $V_{R}^\pm$.

\smallskip\no 
Suppose $(x,y) \in V_{R}^+$ and the constants $0<m<1<M$ be as assumed in Remark \ref{r:filtration estimate}. Then \begin{align}\label{e:BS2}
	\log m+ {d_i}\log|y|< \log|\pi_2 \circ \h_i(x,y)| < \log M+{d_i}\log|y|
\end{align}
for every $1 \le i \le n_0$. Let $h \in \G_k$. Then $h=\h_{j_1} \circ \cdots \circ \h_{j_k}$ where $1 \le j_i \le n_0$ for every $1 \le i \le k$ and the degree of $h$ (denoted by $d_h$) is given by the product $d_{j_1}\hdots d_{j_n}$. Also, recall $D=d_1+\cdots+d_{n_0}$, the total degree of $\s=\langle \G \rangle$. Now by (\ref{e:VR+}), for $(x,y) \in V_R^+$ and $k \ge 1$,
\begin{align*}
	G_k^+(x,y)=\sum_{h \in \G_k}  \frac{\log|\pi_2 \circ h(x,y)|}{D^k}.
\end{align*}
Hence $G_k^+$ is pluriharmonic on $V_R^+$ for every $k \ge 1$ and
\[
	G_{k+1}^+(x,y)=\sum_{i=1}^{n_0}\sum_{h \in \G_k}\frac{\log|\pi_2 \circ \h_i \circ h(x,y)|}{D^{k+1}}.
\]
 Thus from (\ref{e:BS2}) 
{\small \begin{align*}
	G_{k+1}^+(x,y) \le \bigg(\frac{n_0}{D}\bigg)^{k+1}\log M+\sum_{i=1}^{n_0}\sum_{h \in \G_k}\frac{d_i\log|\pi_2 \circ h(x,y)|}{D^{k+1}} \le \bigg(\frac{n_0}{D}\bigg)^{k+1}\log M+G_k^+(x,y). 
\end{align*}
}
 Similarly, by using the left inequality of (\ref{e:BS2}) we have
\begin{align*}
	 \bigg(\frac{n_0}{D}\bigg)^{k+1}\log m+G_k^+(x,y)\le G_{k+1}^+(x,y) \le  \bigg(\frac{n_0}{D}\bigg)^{k+1}\log M+G_k^+(x,y).
\end{align*}
Since $D \ge 2n_0$, the above inequality reduces to
\begin{align}\label{e:recursive_green}
|G_{k+1}^+(x,y)-G_k^+(x,y)| \le \bigg(\frac{1}{2}\bigg)^{k+1}M_0
\end{align}
where  $M_0=\max\{|\log m|,|\log M|\}$ and $(x,y) \in V_R^+.$
Thus the sequence $\{G_k^+\}$ is uniformly Cauchy on $V_R^+$, and hence it converges uniformly to a pluriharmonic function $G_{\G}^+$ on $V_R^+.$ A similar argument on $V_R^-$, proves the same for $G_{\G}^-.$

\smallskip\no 
\textit{Step 2: }The sequence $\{G_k^\pm\}$ converges uniformly to the pluriharmonic function $G_{\G}^\pm$ on compact subsets of $\U^\pm_{\s}$, respectively. 

\smallskip\no As noted earlier similar arguments work on $\U_{\s}^-$, so we complete the proof only for $\U_{\s}^+.$ Let $C$ be a compact subset of $\U^+_\s.$ By Lemma \ref{l:compact sets on U+}, there exists $\N_C \ge 1$ such that $h(C) \subset V_R^+$ for every $h \in \G_k$, $k \ge \N_C.$ Note that $\G_{k}$ has $n_0^{k}$ elements. Let $C_h=h(C) \subset V_R^+$ (by Lemma \ref{l:compact sets on U+}) for every $h \in \G_{\N_C}.$ Thus for $z \in C$
\[{G_k^+}(z)=\frac{1}{D^{\N_C}}\sum_{h \in \G_{\N_C}}{G_{k-{\N_C}}^+}(h(z)).\]
whenever $k>\N_C.$ Now by \textit{Step 1}, $\{G_{k-{\N_C}}^+\}$ is convergent on every $C_h$. Hence $\{G_k^+\}$ converges uniformly on $C \subset \U_\s^+$ to a pluriharmonic function and this completes \textit{Step 2.} Thus, $G_\G^\pm$ are (respectively) pluriharmonic on the $\U_\s^\pm$.

\smallskip Also $G_\G^\pm$ are pluriharmonic on $\I{\ko_\s^\pm}$, as $G_\G^\pm$ are identically zero in here. Next, for $z_0 \in \C^2$ and for $k \ge 1$, we define the following subsets of $\s$, dependent on $z_0$ as 
\begin{align*}
\s^b(z_0)=\{h \in \s: h(z_0) \in V_R \cup V_R^-\},\; &\G^b_k(z_0)=\s^b(z_0) \cap \G_k  \text{ and } \\ \s^u(z_0)=\{h \in \s: h(z_0) \in  \textsf{int}(V_R^+)\}, \;&\G^u_k(z_0)=\s^{u}(z_0) \cap \G_k.
\end{align*}
Note that by (\ref{e:VR+}), the following inequality about the cardinality of the sets $\G^b_k(z_0)$ and $\G^u_k(z_0)$ is immediate for every $k \ge 1$
\begin{align}\label{e:cardinality}
	\sharp\G^b_{k+1}(z_0) \le  n_0(\sharp\G^b_{k}(z_0)), \; n_0(\sharp\G^u_{k}(z_0)) \le \sharp \G^u_{k+1}(z_0).	
\end{align}
Now for $z_0 \in \U_\s^+$, there exists $k_{z_0} \ge 1$ such that $\G^b_k(z_0)=\emptyset$, for $k \ge k_{z_0}$. Otherwise from Lemma \ref{l:step 4}, there exists $k_{z_0} \ge 1$ such that $h(z_0) \in \ov{V_R}$ whenever $h \in \G^b_k(z_0)$ and $k \ge k_{z_0}$. Thus consider the following sequences of functions $\{G^b_k\}$ and $\{G^u_k\}$ defined on $\C^2$ as
\begin{align} \label{e:bdd unbdd}
G^b_k(z)=\sum_{h \in {\G^b_k(z)}}  \frac{\log^+ \|h(z)\|}{D^k} \text{ and }G^u_k (z)=\sum_{h \in \G^u_{k}(z)}  \frac{\log^+\| h(z)\|}{D^k}.
\end{align}
\begin{rem}\label{r:step 4}
Note that $G_k(z)=G^b_k(z)+G^u_k(z)$ for every $z \in \C^2$ and $k \ge 1$. Also if $C$ is a compact subset of $\C^2$ then $h(C) \subset V_R$ whenever $h \in \G_k^b$, $k \ge k_C$, as obtained in Lemma \ref{l:step 4}.	
\end{rem}

\no\textit{Step 3: }The sequence $\{G^b_k\}$ converges uniformly to zero on every compact subset $C \subset \C^2.$

\smallskip\noindent 
Since $C$ is a compact subset of $\C^2$ from Lemma \ref{l:step 4}, for $k \ge k_C$, we have 
\[\sup \big\{\|h(z)\|: z \in C \text{ and } h \in \G^b_k(z)\big\}\le R.\]
Now, by (\ref{e:cardinality}) for every $z \in \mathbb{C}^2$, $\sharp \G^b_k(z) \le \sharp \G_k= n_0^k.$ Hence by the above claim for $z \in C$ \[G^b_k(z) \le \bigg(\frac{n_0}{D}\bigg)^k \log R \le \bigg(\frac{1}{2}\bigg)^k \log R\to 0\] as $k \to \infty.$ This completes \textit{Step 3}.

\smallskip\no 
\textit{Step 4: }For every $z_0 \in \K_\s^+ \setminus \ko_\s^+$ there exist a constant $\widetilde{M}>1$ and a positive integer $\ell_{z_0} \ge 1$ such that for $k \ge \ell_{z_0},$ 
\begin{align}\label{e:z0 relation}
G^u_k(z_0)-\widetilde{M}\bigg(\frac{n_0}{D}\bigg)^{k+1}\le G^u_{k+1}(z_0) \le \bigg(\frac{n_0}{D}\bigg)^{k+1} \widetilde{M}+ G^u_k(z_0).
\end{align}
Since $z_0 \in \K_\s^+ \setminus\ko_\s^+$ there exists a sequence $\{h_n\} \subset \s$ such that $\|h_n(z_0)\| \to \infty$ as $n \to \infty.$ In particular $\s^u(z_0) \neq \emptyset$, i.e.,  by (\ref{e:cardinality}) there exist positive integers $k_{z_0},N_{z_0} \ge 1$ such that
\begin{align*}
	1 \le \sharp \G_{k_{z_0}}^u(z_0) =N_{z_0} < n_0^{k_{z_0}}.
\end{align*} 
Now note that from (\ref{e:VR+}) it follows that $\h_i \circ h(z_0) \in V_R^+$ whenever $h \in \G_{k_{z_0}}^u$ and $1 \le i \le n_0.$ Hence by (\ref{e:cardinality}) for every $k \ge k_{z_0}$ we have
\begin{align}\label{e:cardinality Gk}
n_0^{k-k_{z_0}} N_{z_0} \le \sharp \G_k^u(z_0) < n_0^{k}.
\end{align}
As $0<m<1$, from (\ref{e:cardinality Gk}), (\ref{e:bdd unbdd})  and (\ref{e:BS2}) it follows that for every $k \ge k_{z_0}$,
\begin{align*}
  G^u_{k+1}(z_0) &\ge \frac{1}{D^{k+1}}\sum_{i=1}^{n_0} \sum_{h \in \G^u_k(z_0)}  \log \|\h_i \circ h(z_0)\| =\sum_{i=1}^{n_0}\sum_{h \in \G^u_k(z_0)}  \frac{\log \|\h_i \circ h(z_0)\|}{D^{k+1}},
  \\ &\ge \sum_{i=1}^{n_0}\sum_{h \in \G^u_k(z_0)}\frac{d_i\log \| h(z_0)\|+\log m}{D^{k+1}}
  =\sum_{h \in \G^u_k(z_0)} \frac{\log \| h(z_0)\|}{D^{k}}+\frac{n_0(\sharp \G^u_k(z_0))}{D^{k+1}}\log m
  \\& \ge G_k^u(z_0)+\bigg(\frac{n_0}{D}\bigg)^{k+1} \log m.
\end{align*}
\no Now from Lemma \ref{l:step 4}, there exists $k'_{z_0} \ge 1$ such that for $h \in \G_k^b(z_0)$, $\|h(z_0)\| \le R$ whenever $k \ge k'_{z_0}.$ Let $\ell_{z_0}=\max\{k_{z_0}, k'_{z_0}\}$ and $B=\max\{\|\h_i(z)\|: z \in \overline{V_{R+1}} , 1 \le i \le n_0\}$. Further, for every $k,l \ge 1$, we introduce the following subsets $\G^u_{k,l}(z_0)$ of $\s$ defined as
\[\G^u_{k,l}(z_0)=\{h_1 \circ h_2: h_1 \in \G_l\text{ and }h_2 \in \G^u_{k}(z_0)\}.\]
Now by (\ref{e:VR+}), $\G^u_{k,l}(z_0) \subset \G^u_{k+l}(z_0)$ and $\sharp\G^u_{k,l}(z_0)=n_{0}^l (\sharp \G^u_{k}(z_0))$ for $k \ge \ell _{z_0}.$ Let $h \in \G^u_{k+1} \sm \G^u_{k,1}$ then $h=\h_i \circ \tilde{h}$ for some $\tilde{h} \in \G^b_k$ and $1 \le i \le n_0.$ Thus by the above assumption, $ |\pi_2 \circ h(z_0)|\le B$. Since $$\sharp (\G^u_{k+1}(z_0) \setminus \G^u_{k,1}(z_0)) \le n_0^{k+1}-n_0 (\sharp \G^u_{k}(z_0))$$ and
\begin{align*}
 G^u_{k+1}(z_0)=\frac{1}{D^{k+1}} \sum_{h \in \G^u_{k,1}(z_0)} \log\| h(z_0)\|+\frac{1}{D^{k+1}}\sum_{h \in \G^u_{k+1}(z_0) \setminus \G^u_{k,1}(z_0)}  {\log \| h(z_0)\|},
 \end{align*}
 we have the following inequations (as before)
 \begin{align*}
 G^u_{k+1}(z_0)&=\frac{1}{D^{k+1}} \sum_{i=1}^{n_0}\sum_{h \in \G^u_{k}(z_0)} \log\| \h_i \circ h(z_0)\|+\frac{1}{D^{k+1}}\sum_{h \in \G^u_{k+1}(z_0) \setminus \G^u_{k,1}(z_0)}  {\log \| h(z_0)\|} \\
 & \le \sum_{i=1}^{n_0}\sum_{h \in \G^u_k(z_0)}  \frac{d_i\log \| h(z_0)\|+\log M}{D^{k+1}}+\frac{1}{D^{k+1}}\sum_{h \in \G^u_{k+1}(z_0) \setminus \G^u_{k,1}(z_0)}  {\log \| h(z_0)\|} \\
& \le G^u_k(z_0)+ \frac{n_0 (\sharp \G^u_{k}(z_0))}{D^{k+1}} \log M +\frac{\sharp (\G^u_{k+1}(z_0) \setminus \G^u_{k,1}(z_0))}{D^{k+1}}\log B\\
&\le G^u_k(z_0)+\bigg(\frac{n_0}{D}\bigg)^{k+1} \widetilde{M}
\end{align*}
where $\widetilde{M}=\max\{|\log M|,|\log B|, |\log m|\}$. This completes the proof of \textit{Step 4}.

\smallskip Since $G^u_k \equiv 0$ on $\ko_\s^+$ and $G^b_k \equiv 0$ on $\U_\s^+$, by  \textit{Steps 2, 3} and \textit{4}, $G_k^+$ converges pointwise to a non-negative function $G_\G^+$ on $\C^2$ which is identically zero on $\ko_\s^+$ and pluriharmonic on $\U_\s^+.$
 
 \smallskip\no 
 \textit{Step 5: }$G_\G^+$ is non-negative, continuous and plurisubharmonic on $\uo_\s^+.$
  
 \smallskip
Let $C$ be a compact subset of $\uo_\s^+$ and $z \in C.$ By the same reasoning as in \textit{Step 5}, there exist $k_z, N_z \ge 1 $ such that 
\begin{align*}
	1 \le \sharp \G_{k_{z}}^u(z) =N_{z} \le n_0^{k_{z}}.
\end{align*}
Since $\mathsf{int} (V_R^+)$ is open, there exists $\delta_z>0$ such that the closed ball $\overline{B(z;\delta_z)}$ is contained in $\uo_\s^+$ and $h(B(z;\delta_z)) \subset \mathsf{int} (V_R^+)$ for every $h \in \G^u_{k_z}(z)$, i.e., for every $\xi \in B(z; \delta_z)$
\[ \G^u_{k_z}(z) \subseteq \G^u_{k_z}(\xi) \text{ and } N_z \le \sharp \G_{k_{z}}^u(\xi) \le n_0^{k_z}.\]
Since $\overline{B(z;\delta_z)}$ is compact, by Lemma \ref{l:step 4}, there exists $k'_z \ge 1$, such that $h(\xi) \in V_R$ whenever $h \in \G^b_k(\xi)$ for every $k \ge k'_z$ and $\xi \in \overline{B(z;\delta_z)}.$ Hence by exactly similar argument as in the proof of \textit{Step 4}, for every $k \ge \max\{k_z,k'_z\}$ 
\begin{align*}
G^u_k(\xi)-\widetilde{M}\bigg(\frac{n_0}{D}\bigg)^{k+1}\le G^u_{k+1}(\xi) \le \bigg(\frac{n_0}{D}\bigg)^{k+1} \widetilde{M}+ G^u_k(\xi)
\end{align*}
where $\widetilde{M}$ is as chosen in \textit{Step 4}. Now for a given $\ep>0$, there exist a positive integer $\ell^u_z \ge \max\{k_z,k'_z\}$ sufficiently large, such that for every $p,q \ge \ell^u_z$ and for every $\xi \in \overline{B(z;\delta_z)}$
\begin{align*}
\|G^u_p(\xi)-G^u_q(\xi)\| \le \bigg(\frac{1}{2}\bigg)^{\ell^u_z} \widetilde{M}\le \frac{\ep}{2}.
\end{align*}
As $C$ is compact, there exists a finite collection $\{z_i \in C: 1 \le i \le N_C\}$ such that $C \subset \cup_i B(z_i; \delta_{z_i}).$ Let $\widetilde{C}=\overline{\cup_i B(z_i; \delta_{z_i})}$, then by Lemma \ref{l:step 4}, there exists $\ell_C^\ep \ge \max\{\ell^u_{z_i}: 1 \le i\le N_C\}$ sufficiently large, such that for every $\xi \in \widetilde{C}$ and for every $p,q \ge \ell^\ep_C$
\begin{align*}
 \|G^b_p(\xi)- G^b_{q}(\xi)\|< \bigg(\frac{1}{2}\bigg)^{\ell^\ep_C-1} \log R< \frac{\ep}{2}.
 \end{align*}
Thus for every $\xi \in C$ and $p,q \ge \ell_C^\ep$
\[ \|G^+_p(\xi)- G^+_{q}(\xi)\|< \ep,\]
i.e., $\{G_k^+\}$ is uniformly Cauchy on the compact set $C.$ Further as $\{G_k^+\}$ is a non-negative, subharmonic and continuous sequence of functions on $\uo_\s^+$, so is $G_\G^+.$

\smallskip\no \textit{Step 6:} $G_\G^+$ is non-negative, non-constant, continuous and plurisubharmonic on $\mathbb{C}^2.$

\smallskip Note that if $\ko_\s^+$ is empty there is nothing to proof. So we assume $\ko_\s^+ \neq \emptyset$. By (\ref{e:recursive_green}), it follows that on $V_R^+$ 
\[ \|G_\G^+-\log |y|\| \le  M_0\sum_{i=0}^\infty \bigg(\frac{n_0}{D}\bigg)^{i}.\]
Thus for $(x,y) \in V_R^+$ with $|y|$ sufficiently large, $G_\G^+(x,y)$ is both positive and non-constant. Now from \textit{Step 5}, $G_\G^+$ is continuous on $\uo_\s^+$ and it is identically zero on $\ko_\s^+$. Hence to establish the continuity of $G_\G^+$, it is sufficient to prove $G_\G^+$ is continuous on $\partial \ko_\s^+=\partial \uo_\s^+.$ 
		
\smallskip \no We will prove it by contradiction, so suppose there exists a sequence $\{z_k\} \subset \uo_\s^+$ such that $G_\G^+(z_k)>c>0$ for every $k \ge 1$ and $z_k \to z_0 \in \partial \ko_\s^+. $ Choose $k_0 \ge 1$, large enough such that 
\[c > \frac{\widetilde{M}}{2^{k_0}}\sum_{i=1}^\infty \bigg(\frac{n_0}{D}\bigg)^{i},\] 
where $\widetilde{M}$ is as obtained in \textit{Step 4}. Also, we may assume for every $k \ge 1$,

\[\sup \big\{\|h(z_k)\|: h \in \G^b_{k_0}(z_k)\big\}\le R.\] 
\textit{Claim:} For every $k \ge 1$, $\G_{k_0}^u(z_k) \neq \emptyset.$ 

\smallskip If not, then $\G_{k_0}^u(z_l)=\emptyset$ for some fixed $l \ge 1$, i.e., $\G_{k}^u(z_l) = \emptyset$ for every $1 \le k \le k_0.$ However, there exists $k_l>k_0$ such that $\G_{k_l}^u(z_l)\neq \emptyset$, as $G_\G^+(z_l)>0$. Let $k'_l$ be the least of all such numbers, i.e., $\G_k^u(z_l) \neq \emptyset$ for every $k \ge k_l'>k_0.$ Thus for every $k \ge k_l$		
\[ G^u_{k}(z_l) \le \sum_{i=k_l}^\infty\bigg(\frac{n_0}{D}\bigg)^{i} \widetilde{M} < c.\] As $G^b_k(z_l) \to 0$, the above thus proves that $ G_\G^+(z_l) < c $, which contradicts the assumption on the sequence $\{z_k\}$. Hence the claim follows. 

\smallskip Since $\G_{k_0}^u(z_k) \subset \G_{k_0}$ for every $k \ge 1$ and $\sharp \G_{k_0}$ is finite, there are only finitely many subsets of $\G_{k_0}$. Thus there exists a subsequence $\{z_{k_n}\}$ of $\{z_k\}$ such that the sets $\G_{k_0}^u(z_{k_n})$ are equal for every $n \ge 1.$ Define the sequence $\{h_l\}$ as $h_l=\h_1^l \circ h$ for some $h \in \G_{k_0}^u(z_{k_n})$. Since $h(z_{k_n}) \in V_{R}^+$, $h_l(z_{k_n}) \in V_{R_l}^+$, where $\{R_l\}$ is as obtained in Lemma \ref{l:filtration}, for every $l,n\ge 1$. As $z_{k_n} \to z_0$, this implies 
$h_l(z_0) \in \overline{V_{R_l}^+}$, in particular it contradicts that $z_0 \in \partial \ko_\s^+=\partial \uo_\s^+.$

\smallskip  Hence for every sequence $\{z_k\} \in \uo^+_\s$ and $z_k \to z_0$, $G^+_\G(z_k) \to 0$, which consequently proves $G_\G^+$ is continuous on $\C^2.$
Finally, as $G_\G^+$ coincides with its upper semicontinuous regularisation of $G_\G^+$ and satisfies the sub-mean value property on $\partial \ko_\s^+$, $G_\G^+$ is plurisubharmonic on $\C^2.$

\smallskip Similarly replicating \textit{Step 3,4,5} and \textit{6} for the semigroup $\s^-$ with $\G^-$ as the generating set, gives that $G_\G^-$ is a plurisubharmonic continuous function on $\C^2.$
\end{proof}

\begin{cor}\label{c:Green constant}
	There exist constants $c_{\G}^\pm \in \mathbb{R}$ such that for $(x,y)\in V_R^\pm$ (respectively).
 $$G_{\G}^+(x,y)=\log |y|+O(1) \text{ and }G_{\G}^-(x,y)=\log |x|+O(1).$$
	\end{cor}
\begin{proof}
It follows directly from the equation (\ref{e:recursive_green}), in the proof Theorem \ref{t:result 1}, i.e., on $V_R^+$
	\[-M_0\sum_{i=1}^k \bigg(\frac{n_0}{D}\bigg)^{i} +\log|y| \le G_k^+(x,y) \le M_0\sum_{i=1}^k \bigg(\frac{n_0}{D}\bigg)^{i} +\log|y|. \]	
	Similarly the analogue to (\ref{e:recursive_green}) on $V_R^-$ gives the result for $G_\G^-.$
\end{proof}
\begin{cor}\label{c:non-zero on uo+}
The functions $G_{\G}^\pm >0$ 	restricted to $\uo^\pm_\s$ (respectively).
\end{cor}
\begin{proof}
	From Corollary \ref{c:Green constant} and Lemma \ref{l:filtration}, for $z \in \uo^+_\s$ there exists $n_z \ge 1$ and $h^z \in \G_{n_z}$ such that $h^z(z) \in V_R^+$ and $G_{\G}^+(h^z(z))>0.$ Then from Corollary \ref{c:result 2} we have
	$\displaystyle G_{\G}^+(z) \ge \frac{1}{D^{n_z}}G^+_{\G}(h^z(z)) >0.$
	\end{proof}
	\begin{rem}\label{r:ko pseudoconcave}
	The above corollary also proves $\ko_\s^+$ is pseudoconcave, provided it is non-empty.
	\end{rem}
\section{Proof of Theorem \ref{t:result 3}}\label{s:4}
Recall from Remark \ref{r:inclusions}, the cumulative (positive and negative) Julia sets are contained in the (positive and negative) Julia sets. Our goal in this section is to prove that these two sets are actually equal, by analysing the supports of the positive $(1,1)$-currents on $\C^2$ defined as
\[ \mu_\G^+ =\frac{1}{2\pi} dd^c G_\G^+ \text{ and }\mu_\G^- =\frac{1}{2\pi} dd^c G_\G^-.\]
Note that the above fact is also true for the dynamics of semigroups of rational functions in $\mathbb{P}^1$, and is proved using Ahlfor's covering lemma in \cite{HM}, a tool not available in higher dimensions. Instead we will use the Harnack's inequality for harmonic functions and the properties of the dynamical Green's functions $G_\G^\pm.$ 
\begin{thm}\label{t:convergence of mu}
	Let $\{\g_k\}$ be the sequence of plurisubharmonic functions on $\C^2$ defined as
	\[\g_k^\pm(z)= \frac{1}{D^k}\sum_{h \in \G_k} G_h^\pm(h^\pm(z))=\frac{1}{D^k}\sum_{h \in \G_k} d_hG_h^\pm(z) .\] Then the sequences $\{\g_k^\pm\}$ converge uniformly to $G_\G^\pm$ (respectively) on compact sets of $\C^2$. 	
\end{thm}
Thus, we first observe Lemma \ref{l:uniform convergence}, which is an extension of the proof of Theorem \ref{t:result 1}. 
\begin{lem}\label{l:uniform convergence}
The sequences $\{G_k^\pm\}$\,---\,as in Section \ref{s:3}\,---\,converge uniformly on compact subsets of $\C^2$ to $G_\G^\pm$, respectively.	
\end{lem} 
\begin{proof}
As before, we only prove the convergence of $\{G_k^+\}$ to $G_\G^+$ by generalising the proof of Lemma 8.3.4 from \cite{UedaBook}. The convergence of $\{G_k^-\}$ will follow likewise. Let $C$ be any compact subset of $\C^2$, then
\begin{itemize}[leftmargin=14pt]
	\item If $C$ is contained in $\ko_\s^+$, $G_k^+(z)=G_k^b(z)$ (as defined in the proof of Theorem \ref{t:result 1}). Hence by Remark \ref{r:step 4}, $G_k^+ \to G_\G^+$ uniformly on $C$;
	
	\smallskip
	\item If $C$ is contained in $\uo_\s^+$, from the proof of \textit{Step 5} of Theorem \ref{t:result 1}, it follows that $G_k^+ \to G_\G^+$ uniformly on $C$.
\end{itemize}
So we assume $C \cap \uo_\s^+ \neq \emptyset$ and $C \cap\ko_\s^+ \neq \emptyset$ and let $z \in C.$ By the above facts, for a given $\ep>0$ there exists $k_1 \ge 1$, such that $|G_k^b(z)|<\ep$ and $\|h(z)\|< R$ for every $h \in \G^b_k(z)$ whenever $k \ge k_1.$ In particular, for $z \in C \cap \ko_\s^+$ and $k \ge k_1$
$|G_k^+(z)-G^+_\G(z)|< {\ep}.$ Recall the sets $\{\uo_k\}$ from Proposition \ref{p:K+ closed} defined as
\[\uo_k=\bigcup_{h \in \G_k} h^{-1}(U_0)\] for every $k\ge 1$. Also from (\ref{e:U-inclusion}), we have $\ov{\uo_k} \subset \uo_{k+1} \subset \ov{\uo_{k+1}}$ and $\uo_\s^+=\cup_{k=1}^\infty\uo_k$. Let 
\[ C_k:=C \cap ({\uo_{k}}\setminus {\uo_{k-1}})\]
for every $k \ge 1$. Since $C \cap \uo_\s^+\neq \emptyset$ and $\ov{\uo_k} \subset \uo_{k+1}$, $C_k$'s are non-empty sets for $k \ge 1$, sufficiently large.  Also $\cup_{k=1}^\infty C_k=C \cap \uo_\s^+$. Now for $k >k_1$ and $z \in C_k$, $h(z) \in V_R \cup V_R^-$ whenever $h \in \G_{p}$, $1 \le p \le k-1$, i.e., $\G_{p}^b(z)=\G_p$ for all $z \in C_k.$ Thus from Lemma \ref{l:step 4} and Remark \ref{r:step 4}, for $k$ large enough, $h(z) \in V_R$ whenever $z \in C_k$ and $h \in \G_{k-1}.$ Let
\[B=\sup\{\|\h_i(z)\|: z \in \overline{V_{R+1}}, 1 \le i \le n_0\}.\]
So for $h \in \G_k$, $\|h(z)\|<B$ whenever $z \in C_k$, $k$ sufficiently large. Let $l > k$ and $z \in C_k$, then
\[ G_l^+(z)=\frac{1}{D^{k}} \sum_{h \in \G_k} G_{l-k}^+(h(z)) \le \frac{n_0^k}{D^k} (\log B+\log M) \le \frac{\widetilde{M}}{2^{k-1}}.\]
where $\widetilde{M}=\max\{|\log M|,|\log B|, |\log m|\}$ and $m,M$ is as obtained in Remark \ref{r:filtration estimate}. Now by continuity of $G_\G^+$ for $\ep>0$ there exists a neighbourhood $W$ of $C \cap \partial \ko_\s^+$ such that 
$|G_\G^+(z)| < \ep/2$ for $z \in W$. Further choose $k_2 \ge k_1$ large enough, such that $\frac{\widetilde{M}}{2^{k}} < \ep/4 $ for every $k \ge k_2$ and 
\[ \widetilde{C}_{k_2}:=\bigcup_{k=k_2}^\infty C_k \subset W.\] 
Then for every $z \in \widetilde{C}_{k_2}$,
$|G_k^+(z)-G_\G^+(z)|< \ep$
whenever $k \ge k_2.$ Note that
\[(C \cap \uo_\s^+) \setminus \widetilde{C}_{k_2}\subset C \cap \ov{\uo_{k_2-1}}\] and  $C \cap \ov{\uo_{k_2-1}}$ is a compact set contained in $\uo_\s^+.$ Hence there $k_3 \ge 1$ such that for every $z \in (C \cap \uo_\s^+) \setminus \widetilde{C}_{k_2}$,
$ |G_k^+(z)-G_\G^+(z)| \le \ep. $
Thus $|G^+_k-G_\G^+|<\ep$ on $C$ for $k \ge \max\{k_1,k_2,k_3\}$.
\end{proof}
\begin{rem} Note that in the proof of Theorem \ref{t:result 1}, we use the pointwise convergence of $G_k^\pm$ to prove $G_\G^+$ is a continuous function. However, proof of Lemma \ref{l:uniform convergence} uses the continuity of $G_\G^+$ crucially, to establish the convergence is uniform on compact subsets of $\C^2$. 	
\end{rem}
Now we complete 
\begin{proof}[Proof of Theorem \ref{t:convergence of mu}] We will show that for a given compact set $C \subset \C^2$ and an $\ep>0$ there exists $k_C$ such that  $|\g_k-G_k|_C<\ep/2$ for every $k \ge  \ku_C$, and use Lemma \ref{l:uniform convergence}. 

\smallskip
As before for the compact subset $C$ by Lemma \ref{l:step 4} and Remark \ref{r:step 4},  there exists $k_1 \ge 1$ such that $\|h(z)\|<R$ whenever $h \in \G^b_k(z)$ for every $z \in C$ and $k \ge k_1.$ Now choose $k_2$ such that $\frac{\widetilde{M}}{2^{k-2}}< \ep$ for $k \ge k_2$ where $\widetilde{M}$ as obtained in Lemma \ref{l:uniform convergence}. Thus, for every $z \in C$ with $h \in \G^b_k(z)$ where $k \ge  \ku_C=\max\{k_1,k_2\}$
\begin{align}\label{e:4.1}
	\frac{\log^+\|h(z)\|}{d_h}<\frac{\widetilde{M}}{2^{k}}< \frac{\ep}{2}.
\end{align}
\textit{Step 1:} For $k \ge \ku_C$, $ {G_h^+(z)}< \frac{\ep}{2}$ for every $z \in C$ and $h \in \G_k^b(z)$, i.e.,
$\displaystyle \bigg|\frac{\log^+\|h(z)\|}{d_h}-G_h^+(z)\bigg|< \ep.$

\medskip
If $z \in K_h^+\cap C$ then $G_h^+(z)=0$. So suppose $z \notin K_h^+\cap C$, i.e., $h(z) \in V_R$. Let $h=h_k \circ \cdots \circ h_1(z)$, where $h_i \in \G$ for every $1 \le i \le k.$ Further let $1 \le l \le k$ and $n_z \ge 1$, the minimum positive integers, such that \[w_z:=h_l \circ \hdots \circ h_1 \circ h^{n_z}(z) \in V_{R}^+.\] In particular, $h_{j} \circ \hdots\circ h_1 \circ h^{n}(z) \in V_R$, whenever $1 \le n \le n_z$ and $1 \le j \le l-1$. Then $\norm{w_z}=\norm{h_l \circ \hdots\circ h_1 \circ h^{n_z}(z)}<B$, where $B$ is as chosen in Lemma \ref{l:uniform convergence}. If $1 \le l<k$, then from (\ref{e:VR+})
\[m\norm{w_z}^{d_{l+1}}\le \norm{h_{l+1}(w_z)} \le M\norm{w_z}^{d_{l+1}}\]
where $d_{l+1}$ is the degree of $h_{l+1}$. As $M>1$, we will consider a more robust bound to the above inequality, i.e., $M\norm{w_z}^{d_{l+1}} < \norm{Mw_z}^{d_l+1}$ and obtain the following
	\[ \log^+{\norm{h_{i+l}\circ \cdots \circ h_{l+1}(w_z)}} < d_{l+i}\hdots d_{l+1}(\log B+\log M).\] 
for $1 \le i \le k-l.$ By continuing to repeat the same argument, we get that for every $j \ge 1$ 
	\[ \log^+{\norm{h^{j+n_z}(z)}} < d_h^{j}.d_{k}\hdots d_{l+1}(\log B+\log M) <d_h^{j}(\log B+\log M),\]
where $d_h=d_k\hdots d_1$ is the degree of $h$. Note that if $l=k$, then the final bound on the above inequality is anyway true. Since $d_h\ge 2^k$, for every $j>1$, 
\[ \frac{\log^+{\norm{h^{j+n_z}(z)}}}{d_h^{j+n_z}} \le \frac{\log B+\log M}{d_h^{n_z}}\le \frac{\log B+\log M}{d_h} \le \frac{\widetilde{M}}{2^{k-1}}< \frac{\ep}{2}, \]
and thus \textit{Step 1} follows by taking the limit of $j \to \infty$ and (\ref{e:4.1}).

\smallskip
\noindent\textit{Step 2: }For $h \in \G^u_k(z)$ and $l \ge 2$
\begin{align}\label{e:4.2} 
	\frac{\log\|h(z)\|}{d_h} -\widetilde{M}\sum_{i=1}^{l-1}  \frac{1}{d_h^i} \le \frac{\log\|h^l(z)\|}{d_h^l} \le \frac{\log\|h(z)\|}{d_h} +\widetilde{M}\sum_{i=1}^{l-1}  \frac{1}{d_h^i}.
\end{align}
We will prove the above by induction. Let $l=2$ then by (\ref{e:VR+})
\begin{align}\label{e:4VR+}
	{{m}^{1+\sum_{i=2}^k d_h(i)}}\norm{h(z)} \le \norm{h^2(z)} \le {M}^{1+\sum_{i=2}^k d_h(i)}\norm{h(z)},
\end{align}
where $d(i)=d_i\cdots d_k$, $2 \le i \le k.$ Now, applying logarithm and dividing by $d_h^2$ to the right inequality of the identity (\ref{e:4VR+}), it follows that
\[\frac{\log\norm{h^2(z)}}{d_h^2} \le\frac{  \widetilde{M}}{d_h}\bigg(\sum_{i=1}^k \frac{1}{d_1\cdots d_i}\bigg)+\frac{\log\norm{h(z)}}{d_h} \le \frac{  \widetilde{M}}{d_h}\bigg(\sum_{i=1}^k \frac{1}{2^i}\bigg)+\frac{\log\norm{h(z)}}{d_h} \le \frac{  \widetilde{M}}{d_h}+\frac{\log\norm{h(z)}}{d_h},\]
as $\widetilde{M}>\max\{|\log m|,|\log M|\}$. A similar argument applied to the left inequality of (\ref{e:4VR+}) along with the above observation gives
\[-\frac{  \widetilde{M}}{d_h}+\frac{\log\norm{h(z)}}{d_h} \le \frac{\log\norm{h^2(z)}}{d_h^2} \le \frac{  \widetilde{M}}{d_h}+\frac{\log\norm{h(z)}}{d_h},\]
which proves (\ref{e:4.2}) for $l=2.$ Now, assume (\ref{e:4.2}) is true for some $l \ge 2$, by above
\[-\frac{  \widetilde{M}}{d_h}+\frac{\log\norm{h^{l}(z)}}{d_h} \le \frac{\log\norm{h^{l+1}(z)}}{d_h^2} \le \frac{  \widetilde{M}}{d_h}+\frac{\log\norm{h^l(z)}}{d_h}.\]
Hence dividing further by $d_h^{l-1}$ and substituting the assumption gives
{\small \[\frac{\log\|h(z)\|}{d_h} -\widetilde{M}\sum_{i=1}^{l}  \frac{1}{d_h^i}\le \frac{\log\norm{h^{l}(z)}}{d_h^l}-\frac{  \widetilde{M}}{d_h^l} \le \frac{\log\norm{h^{l+1}(z)}}{d_h^{l+1}} \le \frac{ \widetilde{M}}{d_h^l}+\frac{\log\norm{h^l(z)}}{d_h^l}\le\frac{\log\|h(z)\|}{d_h} +\widetilde{M}\sum_{i=1}^{l}  \frac{1}{d_h^i},\]
}
which proves the induction hypothesis and hence the \textit{Step 2}. 

\smallskip\no Thus, by taking limit $l \to\infty$ on the identity (\ref{e:4.2}), we have
\[\bigg|G_h^+(z)-\frac{\log\|h(z)\|}{d_h}\bigg|\le \frac{2\widetilde{M}}{d_h}\le  \frac{\widetilde{M}}{2^{k-1}}<\ep.\]
Hence for $z \in C$ and $k \ge \ku_C$
\[\bigg| \g_k^+(z)-G_k^+(z)\bigg| \le \sum_{h \in \G_k} \frac{d_h}{D^k}\bigg|G_h^+(z)-\frac{\log\|h(z)\|}{d_h}\bigg|<\ep. \qedhere\]
\end{proof}
\begin{cor}\label{c:inclusion-2}
	Support of $\mu_\G^\pm$ is contained in the cumulative Julia sets $\J_\s^\pm.$
\end{cor}
\begin{proof}
	Let $\mu_k^\pm=\frac{1}{2\pi} dd^c \g_k^\pm$ then from Lemma 3.6 of \cite{BS2}, it follows that 
	\[\text{ supp }(\mu_k^\pm)= \bigcup_{h \in \G_k} J_h^\pm.\]
Let $S$ be any positive $(1,1)-$form supported in the complement of $\J_\s^+$ then $\mu_k^+ (S)=0$ for every $k \ge 1.$ By Theorem \ref{t:convergence of mu} and Corollary 3.6 of \cite{DemaillyBook}, $\mu_k^+ \to \mu^+_\G$, in the sense of currents, i.e., $\mu^+_\G(S)=0.$ Hence the proof. A similar argument works for $\mu_\G^-.$
\end{proof}
Finally, we are ready to complete
\begin{proof}[Proof of Theorem \ref{t:result 3}] Note that by Corollary \ref{c:inclusion-2}, $dd^c(G_\G^+)=0$ everywhere in the complement of $\J_\s^+.$ Choose any ball ${\mathbb{B}}$ contained in  $\C^2 \setminus \J_\s^+.$ As $G_\G^+$ is continuous on $\overline{\mathbb{B}}$, by uniqueness of solution to the Dirichlet problem it follows that $G_\G^+$ is pluriharmonic on $\mathbb{B}$ and $\C^2 \setminus \J_\s^+.$ 
 
 \smallskip Now suppose $\Ju_\s^+ \setminus \text{supp }(\mu_\G^+) \neq \emptyset$ and $z_0 \in \Ju_\s^+ \setminus \text{supp }(\mu_\G^+)$. Then there exists $r>0$, such that the ball of radius $r$ at $z_0$, $B(z_0;r) \subset \big(\text{supp }(\mu_\G^+)\big)^c.$ Let $0<r'<r$. Since $z_0 \in \Ju_\s^+$, there exists a sequence $\{h_n\} \subset \s$ that is neither locally uniformly bounded nor uniformly divergent to infinity on $B(z_0;r').$ In particular, there exist sequences of points $\{z_n\}$ and $\{w_n\}$  in $B(z_0;r')$ such that 
 \[ \|h_n(z_n)\|  \text{ is bounded and } \|h_n(w_n)\| \to \infty\]
 as $n \to \infty.$ Note that without loss of generality we may assume, the length of $h_n \to \infty$ as $n \to \infty$. Now again, by Lemma \ref{l:step 4} and Remark \ref{r:step 4} the above may be modified further as\,--\,for $n$ sufficiently large,
 \[h_n(z_n) \in V_R \text{ and } h_n(w_n) \in V_{r_n}^+,\]
 where $r_n$ is a sequence of positive real numbers that diverges to infinity as $n \to \infty.$ Hence
 \begin{align}\label{e:t3_1}
 G_\G^+\circ h_n(z_n)< C_0 \text{ and } G_\G^+\circ h_n(w_n) \to \infty.
\end{align}
 Also as $G_\G^+$ is pluriharmonic on $B(z_0;r)$ and plurisubharmonic on $\C^2$, by Corollary \ref{c:result 2} we have $G^+_\G \circ h$ is pluriharmonic on $B(z_0;r)$ for every $h \in \s.$ Now by Harnack's inequality (See Theorem 2.5, \cite[Page 16]{TrudingerBook}), there exists $A>0$, a positive constant dependent on $z_0$, $r$ and $r'$, such that for every harmonic function $u$ on $B(z_0;r)$
\[ \sup_{B(z_0;r')} u(z) \le A \inf_{B(z_0;r')} u(z).\] 
Hence $0\le G_\G^+ \circ h_n(w_n) \le AC_0$ which contradicts (\ref{e:t3_1}). Hence $\text{supp }(\mu_\G^+)=\J_\s^+=\Ju_\s^+$.

\smallskip
Further $\mu_\G^+$ is a current of mass 1, follows from Theorem \ref{t:convergence of mu} and Corollary 3.6 in \cite{DemaillyBook}. Similarly by analysing $G_\G^-$ and $\mu_\G^-$ as above, we have $\text{supp }(\mu_\G^-)=\Ju_\s^-$.
\end{proof}
\begin{rem}
Thus by Proposition 3.2 of \cite{DemaillyBook}, the measure $\mu_\G:=\mu_\G^+ \wedge \mu_\G^-$ is a probability measure compactly supported on the intersection of the positive and negative Julia sets.
\end{rem}

\begin{cor}\label{c:escaping set}
	The Fatou component at infinity of the semigroup $\s$ and $\s^-$, i.e., 
	\begin{align}\label{e:infinity component} 
		\U_\s^\pm=\text{int }\Big(\bigcap_{h \in \s}U_h^\pm\Big).
	\end{align} 
\end{cor}
\begin{proof}
	Let $\mathcal{F}^+_h$ denote the Fatou set corresponding to a $h \in \s.$ Note that $\mathcal{F}^+_h=U_h^+ \cup \mathcal{F}^b_{h}$ where $U_h^+$ is the component at infinity and $\mathcal{F}^b_{h}$ are the Fatou components contained in $K_h^+.$ Similarly $\mathcal{F}^-_h=U_h^- \cup \mathcal{F}^b_{h^{-1}}$ where $U_h^-$ is the component at infinity of $h^{-1}$ whenever $h \in \s.$ By Theorem\ref{t:result 3}, it follows that
	\[\mathcal{F}_\s^\pm=\C^2 \setminus \J_\s^+=\text{int }\Big(\bigcap_{h \in \s}\mathcal{F}_h^\pm\Big)=\text{int }\Big(\bigcap_{h \in \s}(U_h^\pm \cup \mathcal{F}_{h^\pm}^b)\Big).\]
	Hence the components at infinity, corresponding to the dynamics of the semigroup $\s$ and $\s^-$ is given by (\ref{e:infinity component}).
\end{proof}
\begin{rem}\label{r:unique Green}
	Also note, if $\ko_\s^+=\K_\s^+$, then $K_{h_1}^+=K_{h_2}^+$ and by Theorem 5.4 from \cite{Lamy}, there exists $m, n \ge 1$ such that $h_1^{m}=h_2^{n}.$ Thus $\J_\s^\pm=J_h^\pm$ for every $h \in \s.$ In particular from Theorem \ref{t:result 4}, it follows that $G_\G^\pm=G_h^\pm$ for every $h \in \s$, i.e., the Green's function is unique.
\end{rem}
However, the next corollary proves that the positive and negative Green's functions obtained corresponding to the semigroup $\s$ is generally non-unique (i.e., whenever  $\ko_\s^+ \subsetneq \K_\s^+$), as a consequence of Corollary \ref{c:result 2} and Corollary \ref{c:Green constant}. 
\begin{cor}\label{c:non-uniqueness}
If $\ko_\s^+ \subsetneq \K_\s^+$ then the Green's functions $G_\G^\pm$ are non-unique and depends on the generating set $\G.$
\end{cor}
\begin{proof}
Suppose not, i.e., let the positive Green's function be unique corresponding to semigroup $\s$. By Proposition \ref{p:minimal generator}, it follows that $\s$ admits a minimal generating set $\G_0.$ Let $\G_h=\G_0 \cup h$, $h \in \s \setminus \G_0.$ Then $\s=\langle \G_0 \rangle =\langle \G_h \rangle.$ Then by assumption, $G_{\G_0}^+=G_{\G_h}^+$ and thus from Corollary \ref{c:result 2}, we have that 
\[(D_{\G_0}+d_h)G_{\G_0}^+(z)=\sum_{\h_i \in \G_0}G_{\G_0}^+(\h_i(z))+G_{\G_0}^+(h(z))=D_{\G_0}G_{\G_0}^+(z)+G_{\G_0}^+(h(z))\] where $D_{\G_0}$ is the total degree of the generating set $\G_0$ and $d_h$ is the degree of $h \in \s.$ Hence $G_{\G_0}^+(z)=d_h G_{\G_0}^+(h(z)) $, i.e., $G_{\G_0}^+(z)=0$ if $z \in K_h^+.$ But from Corollary \ref{c:non-zero on uo+}, the above implies $K_h^+ \subset \ko_\s^+$ for every $h \in \s$. 
 Suppose $\K_\s^+\setminus \ko_\s^+ \neq \emptyset$ and $z_0 \in \K_\s^+\setminus \ko_\s^+$, then there exist sequence $\{h_n\}$ and $\{\tilde{h}_n\}$ in $S$ such that both the lengths of $h_n$ and $\tilde{h}_n$ goes to infinity as $n \to \infty$. Further by Lemma \ref{l:step 4} there exists $\ku_0 \ge 1$ such that for every $n \ge \ku_0$, $h_n(z_0) \in V_R^+ \text{ and } \tilde{h}_n(z_0) \in V_R.$ Then by Corollary \ref{c:Green constant} 
\[G_{\G_0}^+(\h_1^l \circ h_{\ku_0}(z_0))=\log|\pi_2 \circ \h_1^l \circ h_{\ku_0}(z_0)|+O(1),\]
for every $l \ge 1$. Hence $G_{\G_0}^+(\h_1^l \circ h_{\ku_0}(z_0)) \to \infty$ as $l \to \infty$. Fix $l_1 \ge 1$, sufficiently large such that $G_{\G_0}^+(\h_1^{l_1} \circ h_{\ku_0}(z_0))>\widetilde{B}$, where $\widetilde{B}=\max\{G_{\G_0}^+(z): z \in V_R\}.$ Let $\ku_1 \ge \ku_0$
\[\mathsf{h}_1=\h_1^{l_1} \circ h_{\ku_0} \text{ and } \mathsf{h}_2=\tilde{h}_{\ku_1}\]
such that the degree of $\mathsf{h}_2 =d_{\mathsf{h}_2}>d_{\mathsf{h}_1}=$ degree of $\mathsf{h}_1$. Since we have assumed that $G_{\G_0}^+$ is unique, it follows that
\[(d_{\mathsf{h}_2}-d_{\mathsf{h}_1})G_{\G_0}^+(z_0)=G_{\G_0}^+(\mathsf{h}_2(z_0))-G_{\G_0}^+(\mathsf{h}_1(z_0)),\]
i.e., $G_{\G_0}^+(z_0)<0$, which is a contradiction! Thus $\K_\s^+=\ko_\s^+$.

\smallskip Now, if the negative Green's function is unique, similar argument as above will imply $\ko_\s^-=\K_\s^-$. Hence $K_{h_1}^-=K_{h_2}^-$ for every $h_1,h_2 \in \s$. Now by Remark \ref{r:unique Green}, the positive Green's function will also be unique, i.e., $\ko_\s^+=\K_\s^+$, which is again a contradiction!
\end{proof}
\section{Equidistributed projective currents and proof of Corollary \ref{c:result 5}}\label{s:5}
Recall that every polynomial map $g:\mathbb{C}^2 \to \C^2$, i.e., $g(x,y)=(g_1(x,y),g_2(x,y))$ where $g_1$ and $g_2$ are polynomials in $x$ and $y$, extends as a rational map $\bar{g}$ on $\mathbb{P}^2.$ Further in the homogeneous coordinates of $\mathbb{P}^2$, it is defined as
\[\bar{g}[x:y:z]=\bigg[z^d g_1\Big(\frac{x}{z},\frac{y}{z}\Big):z^d g_2\Big(\frac{x}{z},\frac{y}{z}\Big): z^d\bigg]\] 
where $d=\max\{\,\text{degree of }g_1, \text{ degree of }g_2\,\}.$ Now for any map $h \in \s$, $h^{-1}$ is also a polynomial map. Hence both $h$ and $h^{-1}$ extends as rational maps on $\mathbb{P}^2$, in the homogeneous coordinates. Further the degree of $\pi_2 \circ h$ is strictly greater than $\pi_1 \circ h$ and
$\pi_2 \circ h(x,y)=y^{d_h}+\text{ l.o.t.}$ Hence the indeterminancy point of the rational map $\bar{h}$ in $\mathbb{P}^2$ (for every $h \in \s$) is $I^+=[1:0:0].$ A similar argument gives that the indeterminancy point of $\overline{h^{-1}}$ is $I^-=[0:1:0].$ Let $\overline{\s}$ and $\overline{\s^-}$ be the family of the rational maps on $\mathbb{P}^2$ defined as
\[\overline{\s}=\{\bar{h}: h \in \s\} \text{ and } \overline{\s^-}=\{\overline{h^{-1}}: h \in \s\}.\]
Next, we study the dynamics of the above families in $\mathbb{P}^2$ and generalise a few facts from \cite{Dinh-Sibony}. Note that the line at infinity, except the point $I^+$, i.e., $L_\infty^+=\{[x:y:z] \in \mathbb{P}^2: z=0\} \setminus I^+$  is contracted to $I^-$  by every $\bar{h} \in \overline{\s}.$ Similarly the line at infinity, except the point $I^-$, i.e., $L_\infty^-=\{[x:y:z] \in \mathbb{P}^2: z=0\} \setminus I^-$ is contracted to $I^+$ by every $\bar{h} \in \overline{\s^-}$. Also, $V_R^+$ (respectively $V_R^-$) lies in the basis of attraction of $I^-$ (respectively $I^+$) for every $\bar{h} \in \overline{\s}$ (respectively for every $\bar{g} \in \overline{\s^-}$). Hence $I^-$ is attracting fixed point for every $\bar{h} \in \overline{\s}$ and $I^+$ is attracting fixed point for every $\bar{g} \in \overline{\s^-}.$ Thus, we define the following sets.
\begin{itemize}
	\item $\displaystyle \widetilde{\U}_\s^+= \text{int}\Big(\bigcap_{h \in \s} \widetilde{U}_h^+\Big)$, where $\widetilde{U}_h^+$, where $\widetilde{U}_h^+$ is the basin of attraction of $I^-$ for $\bar{h}.$

	\item $\displaystyle \widetilde{\U}_\s^-= \text{int}\Big(\bigcap_{h \in \s} \widetilde{U}_h^-\Big)$, where $\widetilde{U}_h^-$ , where $\widetilde{U}_h^-$ is the basin of attraction of $I^+$ for $\overline{h^{-1}}.$
\end{itemize}
\begin{prop}\label{p:projective closures}
	The sets $\widetilde{\U}_\s^\pm \cap \C^2=\U_\s^\pm.$ Also, the closure of the sets $\K_\s^\pm$ in $\mathbb{P}^2$ is given by $\overline{\K_\s^\pm}=\K_\s^\pm \cup I^\pm.$
\end{prop}
\begin{proof}
Let $\widetilde{\uo}_\s^\pm$ be the basin of attraction of $I^\mp$ for the family $\overline{\s}$ and $\overline{\s^-}$ in $\mathbb{P}^2 \setminus I^\pm$, i.e.,
 \[\widetilde{\uo}_\s^+=\{\Z\in \mathbb{P}^2 \setminus I^+: \exists\text{ a neighbourhood }W \text{ of }\Z\text{ such that } \overline{h_n}_{|W} \to I^- \text{ for every } \{h_n\}  \subset \s \} \]
and
\[\widetilde{\uo}_\s^-=\{\Z\in \mathbb{P}^2 \setminus I^-: \exists\text{ a neighbourhood }W \text{ of }\Z\text{ such that } \overline{h_n^{-1}}_{|W} \to I^+ \text{ for every } \{h_n\}  \subset \s \}.\]
Observe that by definition, if $(x,y) \in \U_\s^\pm$, then $[x:y:1] \in \widetilde{\uo}_\s^\pm$. In particular $\U_\s^\pm \subset \widetilde{\uo}_\s^\pm \cap\C^2.$ Now for any point $\Z_0=[x_0:y_0:z_0] \in L_\infty^+$, $z_0=0$ and $|y_0| \neq 0$. Hence for every $h \in \s$, $\bar{h}(\Z_0)=\bar{h}[x_0:y_0:0]=[0:1:0]$ is immediate.
	
\smallskip\no {\it Claim: }There exist open sets $W^\pm$ containing $L_\infty^\pm$ which is contained $\widetilde{\uo}_\s^\pm$, respectively.
	
\smallskip\no {\it Case 1:} Suppose $|x_0|<|y_0|$, then choose a neighbourhood $W_{\Z_0}$ of $\Z_0$ such that $|x|<|y|$ for every $\Z=[x:y:z] \in W_{\Z_0}$ and $|z|<R^{-1}|y|$ if $z \neq 0$, where $R>R_\s$, the radius of filtration as in Lemma \ref{l:filtration}. Hence for $\Z \in W_{\Z_0} \setminus L_\infty^+$, $\Z=[x:y:1]$ such that $(x,y) \in V_R^+$, i.e, $\bar{h}(\Z) \to [0:1:0]$ as length of $h$ tends to infinity.

\smallskip\no{\it Case 2:} Otherwise, there exists some $ \alpha>1$ such that $|x_0|<\alpha|y_0|$. Note, we need to choose an appropriate neighbourhood of $\Z_0$ contained in $\widetilde{\uo}_\s^\pm.$ We will do so in the light of Remark \ref{r:modified filter}, which is a consequence of the following modification of Lemma 2.2 from \cite{BS2}.  
\begin{lem}\label{l:refined filter}
	Let $H(x,y)=(y,p(y)-ax)$ where $p$ is a polynomial of degree $d_H \ge 2$ and $a \neq 0.$ Also,  let $R_H>0$ be the radius of filtration for $H$ as obtained in Lemma 2.2 of \cite{BS2}. For $R>0$ and $\alpha>1$ we define the following sets as
\[V_{\alpha,R}^+=\{(x,y) \in \C^2: |x|\le \alpha |y|, |y|\ge \alpha^{-1}R\} , \;V_{\alpha,R}^-=\{(x,y) \in \C^2: |x|\ge\alpha |y|, |y|>\alpha^{-1}R\}\]
and
\[\widetilde{V}_{\alpha,R}^-=\{(x,y) \in \C^2: |y|\le \alpha |x|, |x|\ge \alpha^{-1}R\} , \;\widetilde{V}_{\alpha,R}^+=\{(x,y) \in \C^2: |y|\ge \alpha |x|, |x|>\alpha^{-1}R\}.\]
Then there exists an $R^\alpha>\alpha R_H$ such that $H(V_{\alpha,R^\alpha}) \subset V_{R_H}^+$ and $H^{-1}(\widetilde{V}_{\alpha,R^\alpha}^-) \subset V_{R_H}^-.$
\end{lem}
\begin{proof}
 Note that $$V_{R_H}^+ \subset V_{\alpha,R_H}^+,\; V_{\alpha,R_H}^- \subset V_{R_H}^-,\; V_{R_H}^- \subset \widetilde{V}_{\alpha,R_H}^- \text{ and } \widetilde{V}_{\alpha,R_H}^+ \subset V_{R_H}^+ .$$ Also there exists an $R_{\alpha}>\alpha R_H$, sufficiently large, and constant $C_1>0$ such that for $(x,y) \in V_{\alpha,R_\alpha}^+$, i.e., $|y| =\alpha^{-1}R>\alpha^{-1}R_\alpha $ and $|x|\le R$
\[|\pi_2 \circ H(x,y)|>\alpha^{-d_H}C_1 R^{d_H}-|a|R>\alpha^{-1}R=|\pi_1 \circ H(x,y)|.\]
Similarly there exists an $\widetilde{R}_{\alpha}>\alpha R_H$, sufficiently large, and constant $C_2>0$ such that for $(x,y) \in \widetilde{V}_{\alpha,\widetilde{R}_\alpha}^-$, i.e., $|x| =\alpha^{-1}{R}>\alpha^{-1}\widetilde{R}_\alpha $ and $|y|\le R$
\[|\pi_1 \circ H^{-1}(x,y)|>\alpha^{-d_H}C_2R^{d_H}-|a|^{-1}R>\alpha^{-1}R=|\pi_2 \circ H^{-1}(x,y)|.\]
Let $R^\alpha$ be the maximum of $R_\alpha$ and $\widetilde{R}_\alpha $. Then $H(V^+_{\alpha,R^\alpha}) \subset V_R^+$ and $H^{-1}(\widetilde{V}^-_{\alpha,R^\alpha}) \subset V_R^-.$
\end{proof}
\begin{rem}\label{r:modified filter} By a similar technique as in the proof of Lemma \ref{l:filtration}, the above further assures that $R^\alpha>\alpha R_\s$, the radius of filtration of the semigroup $\s$, such that $h(V_{\alpha,R^\alpha}^+) \subset V_R^+$ and $h^{-1}(\widetilde{V}_{\alpha,R^\alpha}^-) \subset V_R^-$ for every $h \in \s.$
\end{rem}
We now choose a neighbourhood $W_{\Z_0}$ of $\Z_0$ such that $|x|<\alpha|y|$ for every $\Z=[x:y:z] \in W_{\Z_0}$ and $|z|{R^\alpha}<\alpha|y|$ if $z \ne 0$, where $R^\alpha$ is as obtained in Remark \ref{r:modified filter}. Hence for $\Z \in W_{\Z_0} \setminus L_\infty^+$, $\Z=[x:y:1]$ where $(x,y) \in V^+_{\alpha,R^\alpha}$, i.e, $\bar{h}(\Z) \to [0:1:0]$ as length of $h$ tends to infinity. 

\smallskip By similar arguments for $h^{-1}$, $h \in \s$ there exists an open set $W^-$ containing $L_\infty^-$ such that $W^- \subset  \widetilde{\uo}_\s^-. $ Further note that for $\Z \in \widetilde{\uo}_\s^\pm \setminus L_\infty^\pm$, $\Z=[x:y:1]$ such that $(x,y) \in \U_\s^\pm.$ Since $\mathbb{P}^2=\C^2 \sqcup L_\infty^\pm \sqcup I^\pm$, $\widetilde{\uo}_\s^\pm \cap\mathbb{C}^2 =\widetilde{\uo}_\s^\pm \setminus L_\infty^\pm = \U_\s^\pm.$ 

\smallskip Finally, as a consequence of Corollary \ref{c:escaping set} and Proposition 5.5 of \cite[Page 28]{Dinh-Sibony}\,---\,which implies $\widetilde{U}_h^\pm \cap\C^2=U_h^\pm$ for every $h \in \s$\,---\,we can write \[\widetilde{\uo}_\s^\pm \setminus L_\infty^\pm=\widetilde{\uo}_\s^\pm \cap \C^2=\U_\s^\pm=\text{int} \bigcap_{h \in \s} {U}_h^\pm=\text{int} \bigcap_{h \in \s} (\widetilde{U}_h^\pm \cap \C^2)=\text{int} \bigcap_{h \in \s} (\widetilde{U}_h^\pm \setminus L_\infty^\pm)\] 
But $L_\infty^\pm \subset W^\pm \subset \widetilde{\uo}_\s^\pm$ and $\widetilde{\uo}_\s^\pm \subset \widetilde{U}_h^\pm$ for every $h \in \s$, i.e., $W^\pm$ is contained in the interior of $\Big(\bigcap_{h \in \s}\widetilde{U}_h\Big)$. Hence the above identity reduces to
\begin{align}\label{e:projective}\U_\s^\pm=\widetilde{\uo}_\s^\pm \setminus L_\infty^\pm=\Big(\text{int} \bigcap_{h \in \s} \widetilde{U}_h^\pm\Big) \setminus L_\infty^\pm=\Big(\text{int} \bigcap_{h \in \s} \widetilde{U}_h^\pm\Big) \cap \C^2=\widetilde{\U}_\s^\pm \cap \C^2.
\end{align}
Now, since $L_\infty^\pm \subset \widetilde{\U}_\s^\pm$, it follows that $\overline{\K_\s^\pm} \subset \K_\s^\pm \cup I^\pm.$ Also, $K_h^\pm \subset \K_\s^\pm$ for every $h \in \s$ and by Proposition 5.8 in \cite[Page 29]{Dinh-Sibony}, $I^\pm \in \overline{K_h^\pm}$. Hence $I^\pm \in \overline{\K_\s^\pm}$, which completes the proof.
\end{proof}
\begin{rem}\label{r:projective currents}
As a consequence of Proposition \ref{p:projective closures}, it follows that the basins of attraction of $I^\pm$ for the families $\overline{\s}$ and $\overline{\s^-}$ are $\widetilde{\U}_\s^\pm$, respectively. Further, the closure of the positive and negative Julia sets $\J_\s^\pm$ in $\mathbb{P}^2$, i.e.,  $\overline{\J_\s^\pm}=\J_\s^\pm \cup I^\pm.$ Hence from Skoda-El-Mir extension Theorem (see \cite{DemaillyBook}), the $(1,1)$-currents $\mu_\G^\pm$ extends by $0$ to positive closed $(1,1)$-currents (will also be denoted by $\mu_\G^\pm$) on $\mathbb{P}^2$. Now as $G_\G^\pm$ are the logarithmic potential of $\mu_\G^\pm$ restricted to $\C^2$\,---\,from the observation in Example 3.7 in \cite{Dinh-Sibony}\,---\,the functions $g_\G^\pm(z)=G_\G^\pm(z)-\frac{1}{2}\log(\|z\|^2+1)$ are the quasi-potentials corresponding to the currents $\mu_\G^\pm$ on $\mathbb{P}^2$ (respectively). 
\end{rem}
\begin{rem}\label{r:Projective Green}
Note that the functions $g_\G^\pm(z)$ is uniformly bounded and pluriharmonic on $V_{R_\s}^\pm$ (respectively) from Corollary \ref{c:Green constant}. Hence for every $k\ge 1$, on $\U_k:= \bigcap_{h \in \G_k}h^{-1}(V_R^+)$
\[G_\G^+(z)- \frac{1}{2}\log(\|z\|^2+1)=\frac{1}{D^k}\sum_{h \in \G_k}G_\G^+(h(z))-\frac{1}{2}\log(\|z\|^2+1) \text{ is pluriharmonic.}\]
 Since $\widetilde{\U}_\s^+=\U_\s^+\cup L_\infty^+$ is an open set containing $L_\infty^+$,  $g_\G^+(z)$ extends as a pluriharmonic function on $\widetilde{\U}_\s^+$. A similar arguments gives $g_\G^-(z)$ extends as a pluriharmonic function on $\widetilde{\U}_\s^-$.
\end{rem}

Next, we prove a generalisation of Theorem 6.6 from \cite{Dinh-Sibony} in our setup.\begin{prop}\label{p:projective current}
Let $\V$ be  a neighbourhood of $I^-$ and $\{S_k\}$, a sequence of positive $(1,1)$-closed currents of mass 1 in $\mathbb{P}^2$ such that each $S_k$, $k \ge 1$, admits a quasi-potential $u_k$, satisfying $0<|u_k| \le A$ (a constant) on $\V$. Then there exists $c>0$ such that for every $\mathcal{C}^2$ test $(1,1)$-form $\phi$ on $\mathbb{P}^2$
\[|\langle \mu_k^*-\mu_\G^+, \phi \rangle| \le \frac{ck}{2^k}\|\phi\|_{\mathcal{C}^2}, \text{ i.e., }\mu_k^*:=\frac{1}{D^k} \sum_{h \in \G_k} \bar{h}^*(S_k) \to \mu_\G^+.\]
\end{prop}
\begin{proof}
 Note that we may assume that $dd^c u_k=S_k -\omega_{\text{FS}}$, where $w_{\text{FS}}$ is the Fubini-Study $(1,1)$-form on $\mathbb{P}^2$ and $\V$ is a sufficiently small neighbourhood of $I^-$ contained in $\widetilde{\U}_\s^+$. In particular, $\V \cap \C^2 \subset V_R^+.$ Then for $z \in \V \cap V_{R_\s}^+ $, and by the identities in Section \ref{s:4}, $g_h^+(z)\le \widetilde{M}$
 \[ g_k^+(z):=\widetilde{G}_k^+(z)-\frac{1}{2}\log(\norm{z}^2+1) \le \widetilde{M} \text{ for every }h \in \G_k, ~k \ge 1.\]
 Thus by Remarks \ref{r:projective currents} and \ref{r:Projective Green}, the quasipotentials $u_k-g_h^+$, $u_k-g_k^+$ and $u_k-g_\G^+$ are uniformly bounded on $\V$, by $\widetilde{M}+A$, with $g_k^+ \to g_\G^+$ uniformly on $\V$. Also by the proof of Lemma \ref{l:refined filter}, let $\alpha \ge 1$ be such that $ \text{supp}(S_k)\cap \C^2 \subset V_{R_\alpha} \cup \widetilde{V}_{R_\alpha}^-$, for every $k \ge 1$. Hence $\text{supp}(h^*(S_k))\cap \C^2\subset V_{R_\alpha} \cup V_{R_\alpha}^-$ for very $h \in \s$. Thus we refine $\V$ again, so that $\V \cap \C^2 \subset V_{R_{\alpha}}^+$ and by continuity
  \[
\text{supp}(\mu_k^*-\mu_\G^+)\cap \V=\emptyset,~\text{supp}(\mu_k^*-\mu_k^+)\cap \V=\emptyset,~\text{supp}\Big(\frac{1}{d_h}\bar{h}^*(S_k)-\mu_h^+\Big)\cap \V=\emptyset.
  \]
 for every $h \in \s$ and $k \ge 1.$ Hence 
 \begin{align}
 |\langle \mu_k^*-\mu_\G^+, \phi \rangle|_{\mathbb{P}^2}=|\langle \mu_k^*-\mu_\G^+, \phi \rangle|_{\mathbb{P}^2 \setminus \V}\text{ and }|\langle \mu_k^*-\mu_k^+, \phi \rangle|_{\mathbb{P}^2}=|\langle \mu_k^*-\mu_k^+, \phi \rangle|_{\mathbb{P}^2 \setminus \V}	
 \end{align}
Also $\h_i^{-1}(\mathbb{P}^2 \setminus \V) \subset \mathbb{P}^2 \setminus \V$ as  $I^+$ is a super attracting fixed point for every $\h_i$, $1 \le i \le n_0.$ Hence the $\mathcal{C}^1$-norm of every $\overline{h^{-1}}$ is bounded by $\M^k$ for some $\M>0$, whenever $h \in \G_k$. Since $u_k-g_h^+$, $u_k-g_k^+$ and $u_k-g_\G^+$ are d.s.h. functions in $\mathbb P^2$,  by \cite[Lemma 3.11]{Dinh-Sibony} the DSH-norm (see \cite[Section 3]{Dinh-Sibony} for definition) of $u_k-g_h^+$, $u_k-g_k^+$ and $u_k-g_\G^+$ are uniformly bounded for very $k \ge 1$ and $h \in \s.$ Hence by \cite[Lemma 3.13]{Dinh-Sibony} there exists a constant $C_0 \ge 0$ such that
{\small\begin{align}\label{e:DSconvergence} 
|\langle \mu_k^*-\mu_\G^+, \phi \rangle|_{\mathbb{P}^2 \setminus \V} ={D^{-k}}  \sum_{h \in \G_k}|\langle u_k-g_\G^+, dd^c(\phi \circ \overline{h^{-1}}) \rangle| \le C_0 (n_0D^{-1})^k (1+\log^+ \M^{4k})\norm{\phi}_{\mathcal{C}^2}
\end{align}
}
 which completes the proof.
\end{proof}
\begin{rem}\label{r:NADS}
Furthe Lemma 3.11 and Lemma 3.13 of \cite{Dinh-Sibony} also gives that for $C
_0>0$, where $C_0$ is as obtained in the above proof of Proposition \ref{p:projective current},
{\small \begin{align*}
&\bullet ~|\langle \mu_k^*-\mu_k^+, \phi \rangle|_{\mathbb{P}^2 \setminus \V}={D^{-k}}  \sum_{h \in \G_k}|\langle u_k-g_h^+, dd^c(\phi \circ \overline{h^{-1}}) \rangle| \le C_0 (n_0D^{-1})^k (1+\log^+ \M^{4k})\norm{\phi}_{\mathcal{C}^2},\\
&\bullet ~\bigg|\langle \frac{1}{d_h}\bar{h}^*(S_k)-\mu_k^+, \phi \rangle\bigg|_{\mathbb{P}^2 \setminus \V}={d_h^{-1}} |\langle u_k-g_h^+, dd^c(\phi \circ \overline{h^{-1}}) \rangle| \le C_0 d_h^{-1} (1+\log^+ \M^{4k})\norm{\phi}_{\mathcal{C}^2}.
\end{align*} }
\end{rem}

Now, as a direct consequence of the above proposition we observe the following.
\begin{cor}\label{t:result 4}
Let $S$ be a closed positive $(1,1)$-current in $\mathbb{P}^2$ of mass 1, such that support of  $S$ does not contain the point $[0:1:0]$ and $\bar{h}$ be the extension of $h$ to $\mathbb{P}^2$, $h \in \s.$ Then 
\[\lim_{k \to \infty}\frac{1}{D_\G^n} \sum_{h \in \G_k} \bar{h}^*(S) \to \mu_\G^+. \]
\end{cor}
\begin{proof} Note that if $S$ is an $(1,1)$ positive closed current of mass 1 in $\mathbb{P}^2$, and let $u$ be a quasi-potential associated to $S$, i.e., $u$ is a quasi p.s.h function and $dd^c u=S-\omega_{FS}.$ Then $u$ is bounded on a neighbourhood of $I^-$ and by Proposition \ref{p:projective current}, the proof follows.
\end{proof}
\begin{rem}\label{r:mu_G-}
Since the analogue of Proposition \ref{p:projective current}, is true for the current $\mu_\G^-$ as well, if $S$ is a positive $(1,1)$ current of mass 1 on $\mathbb{P}^2$ then 
$\displaystyle	\lim_{k \to \infty}\frac{1}{D_\G^k} \sum_{h \in \G_k} \bar{h}_*(S) \to \mu_\G^-.$
\end{rem}
Thus, we conclude the uniqueness of $\mu_\G^\pm$ from Corollary \ref{t:result 4}.
\begin{proof}[Proof of Corollary \ref{c:result 5}] Let $S$ be a positive $(1,1)$ closed current supported on $\Ju_\s^+$ satisfying property (\ref{e:functorial current}). Then $S$ extends across $I^+$ by zero as a closed $(1,1)$ current of mass $1$ on $\mathbb{P}^2$ that does not intersect $I^-.$ Thus by Theorem \ref{t:result 4}, it follows that on $\C^2$
\[	\lim_{k \to \infty}\frac{1}{D_\G^k} \sum_{h \in \G_k} h^*(S) \to \mu_\G^+.\]
But from (\ref{e:functorial current}), $\frac{1}{D_\G^k} \sum_{h \in \G_k} h^*(S)=S$. Hence $S=\mu_\G^+.$ A similar argument for $\mu_\G^-.$
\end{proof}
Finally, we end this section with the interpretation of Corollary \ref{t:result 4} for algebraic varieties.
\begin{cor}\label{c:result 4}
Let $S$ be an affine algebraic variety of codimension 1 in $\mathbb{C}^2$, then there exist non-zero constants $\mathbf{c}^\pm >0$ such that 
\[\lim_{k \to \infty}\frac{1}{D_\G^n} \sum_{h \in \G_k} h^*[S] \to \mathbf{c}^+\mu_\G^+ \text{ and } \lim_{k \to \infty} \frac{1}{D_\G^n} \sum_{h \in \G_k} h_*[S] \to \mathbf{c}^-\mu_\G^-.\]
\end{cor}
\begin{proof}
Let $S$ be an algebraic variety of codimension $1$ in $\C^2$, i.e., $S=\{(x,y) \in \C^2: p(x,y)=0\}$ where $p$ is a polynomial of degree at least 1. Let $p(x,y)=\sum_{a,b \in \mathbb{N}}p_{ab}x^ay^b$ such that $p_{ab}=0$ whenever $a$ and $b$ is greater than some fixed positive integer. The degree of $p$ is $\mathbf{d}_p=\max\{a+b: p_{ab} \neq 0\}.$

\smallskip\noindent \textit{Case 1:} Let $p(x,y)=c_py^{\mathbf{d}_p}+\text{l.o.t.}$ then $S$ is a quasi-projective variety of $\mathbb{P}^2$ of codimension $1$ and $\bar{S}$ extends to $\mathbb{P}^2$ as an analytic variety, that does not contain $I^-.$ Hence the current of integration of $[\bar{S}]$ is a closed positive $(1,1)$ current of finite mass, say $\mathbf{c}^+$ (see \cite[Page 140]{DemaillyBook}). Thus from Theorem \ref{t:result 4} it follows that
$\displaystyle \lim_{k \to \infty}\frac{1}{D_\G^k} \sum_{h \in \G_k} h^*[S] \to \mathbf{c}^+\mu_\G^+.$

\noindent \textit{Case 2:} For any polynomial $p$, a generalisation of Proposition 4.2 in \cite{BS2} (or Proposition 8.6.7 in \cite{UedaBook}) gives that there exists $k \ge 1$ such that $p_h=p \circ h$ for all $h \in \G_k$, is as in \textit{Case 1}, i.e.,

\smallskip\no 
\textit{Claim: }There exists $k_0 \ge 1$ such that for every $h \in \G_k$ and $k \ge k_0$,
\begin{align}\label{e:polynomial form}
p \circ h(x,y)=c_{p_h} y^{\mathbf{d}_{\tilde{p}_h}}+ \text{l.o.t.}
\end{align} 
For a positive integer $i \ge 1$, let $\lambda_i(p)=\max\{a+ib: p_{ab} \neq 0\}$ and $$\rho_i(p)= \{(a,b): a+ib=\lambda_i(p) \text{ and } p_{ab} \neq 0\},$$ i.e., the terms in the leading part of the polynomial $p$ with weight $i.$  Let $H$ be a generalised H\'{e}non map of the form (\ref{e:ghm}) of degree $d_H$. We first note the following result, which is a rephrasing of Lemma 8.6.5 from \cite{UedaBook}. 

\smallskip\no\textbf{Result.} \textit{For a polynomial $p(x,y)=\sum_{a,b \in \mathbb{N}} p_{ab} x^a y^b$ the number of elements in the leading term of $p \circ H$ in $i$ weight, $i \ge 2$, i.e., $\sharp \rho_{i}(p \circ H)$ satisfies the following inequality
\[\sharp \rho_{i}(p \circ H) \le 1+ \frac{\sharp \rho_{d_H}(p)-1}{i}.\]
}
So if $\h$ is a map of the form (\ref{e:fghm}), of degree sufficiently large, it follows from the above result, that the number of leading terms in any weight $i$, $i \ge 2$ of the polynomial $p \circ \h$ is 1. Now if $H$ is a generalised H\'{e}non map of form (\ref{e:ghm}), then the degree of $p \circ \h \circ H(x,y) $ is $\lambda_{d_H}(p \circ \h)$ and $\rho_{d_H}(p \circ \h)=c_{p_\h} x^a y^b$ where $a+d_H b=\lambda_{d_H}(p \circ \h).$ Hence $p \circ \h\circ H(x,y)=c_{p_\h}c'_Hy^{\lambda_{d_H}}+ \text{l.o.t.}$ Since for very $h \in \G_k$ degree of $h$ is greater than $2^k$, from Lemma 8.6.6 of \cite{UedaBook} there exists $k_0 \ge 1$ such that the polynomial $p \circ h$ has the desired form (\ref{e:polynomial form}) whenever $h \in \G_k$, $k \ge k_0.$
 
\smallskip\no Thus, \textit{Case 1} applied to every polynomial $p_h$, $h \in \G_{k_0}$, proves Corollary \ref{c:result 4} for $\mu_\G^+.$
\end{proof}	

\section{Green's functions for non-autonomous sequences in $\s$}\label{s:6}
  Let $\seq{h} \subset \s$, where $\s$ is the semigroup of H\'{e}non maps as defined in (\ref{e:S}). Recall that to study dynamics of the sequence  $\seq{h}$, one needs to study the behaviour of the sequences $\{h(k)\}$ and $\{h^{-1}(k)\}$ defined as
  \[h(k):=h_k \circ \cdots \circ h_1 \text{ and }h^{-1}(k)=h_k^{-1} \circ \cdots \circ h_1^{-1}.\] 
  Now as each $h_i, i \ge 1$, is generated by the finitely many elements of $\G$, there exist a sequence $\seq{\tilde{h}} \subset \G$ and a sequence $\seq{n}$ of positive integers such that for every $k \ge 1$
  \begin{align}\label{e:equiv}
  h(k)=h_k \circ \cdots \circ h_1=\tilde{h}_{n_k}\circ \cdots \circ \tilde{h}_1=\tilde{h}(n_k). 	
  \end{align}
  Hence with abuse of notation, we will assume that $\seq{h} \subset \G$, i.e., the elements of the sequence $\seq{h}$ varies within the finite collection $\G=\{\h_i: 1 \le i \le n_0\}.$ Also analogue of the (positive and negative) escaping sets and the bounded sets for the sequence $\seq{h}$ is defined as
 \[
  	\U^+_{\seq{h}}=\{z \in \C^2: h(k)(z) \to \infty \text{ as }k \to \infty\},
  	\U^-_{\seq{h}}=\{z \in \C^2: h^{-1}(k)(z) \to \infty \text{ as }k \to \infty\}
  \]
and
\begin{align*}
  	\K^+_{\seq{h}}&=\{z \in \C^2: \{h(k)(z)\} \text{ is bounded}\}, \;\K^-_{\seq{h}}&=\{z \in \C^2:\{h^{-1}(k)(z)\} \text{ is bounded}\}.
  \end{align*}
Since $h(k) \in \G_k$, by Remark \ref{r:filtration estimate} the following inequality holds for every $(x,y) \in V_{R}^+$, $R>R_\s$ (sufficiently large)
 \begin{align}\label{e: NA filtration+}
	m \abs{y}^{d_k}<\norm{h_k(x,y)}=\abs{\pi_2 \circ h_k(x,y)}<M\abs{y}^{d_k},
\end{align}
 where $d_k$ is the degree of $h_k.$ Also for $(x,y) \in V_{R}^-$
 \begin{align}\label{e:NA filtration-}
	m \abs{x}^{d_k}<\norm{h_k^{-1}(x,y)}=\abs{\pi_1 \circ h_k^{-1}(x,y)}<M\abs{x}^{d_k}.
 \end{align}
 Also $h(k)({V_{R}^+}) \subset \I{V_{R_k}^+}$ and $h(k)^{-1}({V_{R}^-}) \subset \I{V_{R_k}^-}$ where $R_k \to \infty$ as $k \to \infty.$ 
\begin{rem}\label{r:NA K+ closed} 
 Thus $\I{V_R^\pm} \subset \U_{\seq{h}}^\pm$. Also we enlist the following observations on the escaping and non-escaping sets, which follows from the same arguments as in Proposition \ref{p:K+ closed}.
 \begin{itemize}[leftmargin=14pt]
 	\item $\displaystyle  \U^+_{\seq{h}}=\bigcup_{k=0}^\infty h(k)^{-1}\Big(\I{V_R^+}\Big)$ and $\displaystyle  \U^-_{\seq{h}}=\bigcup_{k=0}^\infty (h^{-1}(k))^{-1}\Big(\I{V_R^-}\Big).$
 	\item $\K_{\seq{h}}^\pm$ are closed subsets of $\C^2$ and $\K_{\seq{h}}^\pm=\C^2 \setminus \U_{\seq{h}}^\pm.$
 	\item $\U_{\s}^\pm \subset \U_{\seq{h}}^\pm \subset \uo^\pm_{\s}$ and  $\ko_{\s}^\pm \subset \K_{\seq{h}}^\pm \subset \K^\pm_{\s} \subset V_R \cup V_R^\mp.$ 
 \end{itemize}
\end{rem}
 Now as in Section \ref{s:2}, consider the following sequences of plurisubharmonic functions on $\C^2$
\begin{align}\label{e:NA Green sequence}	
	\Gr_k^+(z)=\frac{1}{\du_k} \log^+\|h(k)(z)\| \text{ and }\Gr_k^-(z)=\frac{1}{\du_k}  \log^+\|h^{-1}(k)(z)\|,
 \end{align}
 where $\du_k=d_1\hdots d_k$ is the degree of $h(k).$ Then, we have an analogue to Theorem \ref{t:result 1} here. 
 \begin{thm}\label{t: NA Green function}
	The sequences of functions $\{\Gr_k^\pm\}$ converges pointwise to a plurisubharmonic continuous functions $\Gr_{\seq{h}}^\pm$ on $\C^2$, respectively. Further, $\Gr^\pm_{\seq{h}}$ is pluriharmonic on $\U_{\seq{h}}^\pm$ and $\I{\K_{\seq{h}}^\pm}$.
 \end{thm}
The proof of the above theorem and other important results\,---\,obtained in this section\,---\,are essentially revisiting the techniques discussed through sections \ref{s:3}, \ref{s:4} and \ref{s:5}, in the current non-autonomous dynamical setup. Hence the presentations will be mostly brief and sketchy. Also, note that Remark \ref{e:NA Green sequence} and the definition of functions $\seq{\Gr^\pm}$ above, is valid for any non-autonomous sequence $\seq{h}$ of H\'{e}non maps, satisfying the identities (\ref{e: NA filtration+}) and (\ref{e:NA filtration-}).

\begin{proof}
\textit{Step 1:} The sequence of functions $\seq{\Gr^+}$ converges uniformly on compact subsets of $V_{R}^+$ and the sequence of functions $\seq{\Gr^-}$ converges uniformly on compact subsets of $V_{R}^-.$ 

\smallskip\no From the filtration identity (\ref{e: NA filtration+}) it follows that for $(x,y) \in V_{R}^+$
	\begin{align*}
		\Gr_{k-1}^+(x,y)+\frac{\log m}{\du_k} \le \Gr_{k}^+(x,y) \le 	\Gr_{k-1}^+(x,y)+\frac{\log M}{\du_k}.
	\end{align*}
As $\du_k \ge 2^k$ for every $k\ge 1$, we have
\[\big|\Gr_{k-1}^+(x,y)-\Gr_k^+(x,y)\big| \le \frac{\widetilde{M_0}}{2^k},\]	where $\widetilde{M_0}=\max\{|\log m|, |\log M|\}.$ Thus for a given $\epsilon>0$ there exists $m,n \ge 1$, sufficiently large,	$\abs{\Gr_m^+-\Gr_n^+} \le \epsilon$ on $V_R^+$.  A similar argument works on $V_{R}^-$.

\smallskip\noindent\textit{Step 2:} The sequence of functions $\seq{\Gr^+}$ converges uniformly on compact subsets of $\U_{\seq{h}}^+$ and the sequence of functions $\seq{\Gr^-}$ converges uniformly on compact subsets of $\U_{\seq{h}}^-.$

\smallskip\noindent Note that by Remark \ref{r:NA K+ closed}, for a given compact set $C \subset \U_{\seq{h}}^+$, there exists $\ell_C \ge 1$, large enough such that $h(\ell_C)(C) \subset V_R^+$. Thus by similar argument as above, for $(x,y) \in C$
\[\big|\Gr_{k-1}^+(h(\ell_C)(x,y))-\Gr_{k}^+(h(\ell_C)(x,y))\big| \le \frac{\widetilde{M_0}}{2^k}.\]
	Now for a fixed $\ell_0 \ge 1 $ and $k>1$ consider the sequence of functions defined as
\[\Gr_{k}^{\ell_0}(z)=\frac{\du_{\ell_0}}{\du_{k+\ell_0}} \log^+\|h(k+\ell_0)h(\ell_0)^{-1}(z)\|\]
As $h(k+\ell_0)h(\ell_0)^{-1}=h_{k+\ell_0}\circ\cdots\circ h_{1+\ell_0}$, the functions $\{\Gr_k^{\ell_0}\}$ are pluriharmonic on $V_R^+$, by the same argument as for $\{\Gr_k^+\}$. Since $C \subset h(\ell_C)^{-1}(V_R^+)$, also by the filtration identity (\ref{e: NA filtration+})
\[\big|\Gr_{k-1}^{\ell_C}(h(\ell_C)(x,y))-\Gr_{k}^{\ell_C}(h(\ell_C)(x,y))\big| \le \frac{\widetilde{M_0}}{2^k}\]
for every $(x,y)\in C$. Note that 
\[\Gr_{k+\ell_C}^+(x,y)=\frac{\Gr_{k}^{\ell_C}\big(h(\ell_C)(x,y)\big)}{\du_{\ell_C}}.\]
Hence for $k\ge 1$, sufficiently large and $(x,y) \in C$
\[\big|\Gr_{k+\ell_C-1}^{+}(x,y)-\Gr_{k+\ell_C}^{+}(x,y)\big| \le \frac{1}{\du_{\ell_C}} \frac{\widetilde{M_0}}{2^k}.\]

Thus $\{\Gr_k^+\}$ converges to a pluriharmonic function on $h(\ell_C)^{-1}(V_R^+)$. Hence the function $\Gr^+_{\seq{h}}$ is pluriharmonic on $\U_{\seq{h}}^+$. A similar proof works for $\Gr^-_{\seq{h}}$ and $\U_{\seq{h}}^-$.

\smallskip\noindent\textit{Step 3:} 
Let $\displaystyle \Gr_{\seq{h}}^\pm:=\lim_{k \to \infty} \Gr_k^\pm$, the pointwise limits of $\{\Gr^\pm_k\}$. Then both the limit functions $\Gr_{\seq{h}}^\pm$ are continuous and plurisubharmonic on $\C^2.$

\smallskip\noindent
To complete the above, we first prove that $\Gr^+_{\seq{h}}$ is continuous on $\C^2$, in particular it is continuous on $\partial \K_{\seq{h}}^+.$ Suppose not, then there exist a point $z_0 \in \partial\K^+_{\seq{h}}$ and a sequence $\{z_n\} \in \U^+_{\seq{h}}$ such that $z_n \to z_0$ such that $\Gr^+_{\seq{h}}(z_n)>c>0$ for every $n \ge 1.$ Let $B=\max\{\log^+\norm{\h_i(z)}: z \in \overline{V_{R}} \text{ and } 1 \le i \le n_0\}.$ 
Also let $\ku_1 \ge 1$, sufficiently large, such that \[\frac{c}{2}>\frac{\widetilde{M}}{2^{\ku_1}}\] where $\widetilde{M}:=\max\{\widetilde{M_0}, B\}$. Further note that there exists $\ku_2 \ge 1$ such that $h(k)(z_n) \in V_{R} \cup V_{R}^+$ for $k \ge \ku_2.$ If not, then there exists a subsequence $\{k_n\}$ of positive integers diverging to infinity and a subsequence $\{z_{l_n}\}$ of $\{z_n\}$ such that $h(k_n)(z_{l_n}) \in V_{R}^-$, i.e., $z_{l_n} \in h(k_n)^{-1}(V_{R}^-).$ Hence $z_{l_n} \notin \I{V_{R_{k_n}}}$. As $R_{k_n} \to \infty$, this would mean $\norm{z_{l_n}}\to \infty$, which is a contradiction! 

\smallskip\noindent
\textit{Claim:} The sequence $\{h(k)(z_{n})\} \in V_{R}^+$ whenever $k \ge \max\{\ku_1,\ku_2\}$. Suppose not, then there exist $k_l \ge \max\{\ku_1,\ku_2\}$ and $z_l \in \seq{z}$, such that $h(k_l)(z_{l}) \in V_{R}.$ Also from (\ref{e: NA filtration+}), $h(k)(z_l) \notin V_{R}^+$ for every $k \le k_l.$ Let $\ku_l >k_l$ be the minimum positive integer such that $h(\ku_l)(z_l) \in V_{R}^+$, i.e., $\norm{h(\ku_l)(z_l)} \le B$ and $h(k)(z_l) \in V_{R}^+$ for every $k \ge \ku_l.$ Hence for every $k >\ku_l$
	\[ \Gr_k^+(z_l)\le \frac{\log^+\norm{h(\ku_l)(z_l)}}{\du_{\ku_l}}+\widetilde{M_0}\sum_{i=\ku_l+1}^k\frac{1}{\du_i} \le \widetilde{M}\sum_{i=\ku_l}^k\frac{1}{\du_i}.\]
Since $\du_i \ge 2^i$ and $\ku_l-1 \ge \ku_1$, the above simplifies to
	\[ \Gr_k^+(z_l)\le\widetilde{M}\sum_{i=\ku_l}^k\frac{1}{2^i} \le \frac{\widetilde{M}}{2^{\ku_l-1}}\le \frac{\widetilde{M}}{2^{\ku_1}}<\frac{c}{2}.\]	So $\Gr_{\seq{h}}^+(z_l)<c$, which is a contradiction to the assumption. Thus the claim follows.

	\smallskip\noindent 
As $z_0 \in \partial \K^+_{\seq{h}} \subset \K^+_{\seq{h}}$ (since it is closed), by Lemma \ref{l:step 4} and Remark \ref{r:step 4} there exists $\ku_0 \ge 1$ such that $\norm{h(k)(z_0)} \le R$ for every $k \ge \ku_0.$ Also, from the above \textit{Claim} and (\ref{e: NA filtration+}) we may fix a $k_0 > \max\{\ku_0,\ku_1,\ku_2\}$ such that $h(k_0)(z_n) \in \I{V_{R+1}^+}.$ Since $z_n \to z_0$, by continuity of $h(k_0)$ we have $h(k_0)(z_0) \in V_{R+1}^+$, i.e., $\norm{h(k_0)(z_0)} \ge R+1 >R$, which is not possible (as $z_0\in \K^+_{\seq{h}}$). Thus $\Gr^+_{\seq{h}}$ is continuous on $\C^2.$

	\smallskip 
Now $\Gr^+_{\seq{h}}$ is pluriharmonic on $U_{\seq{h}}^+$, and is identically zero, i.e, is also pluriharmonic in the interior of $K_{\seq{h}}^\pm$ (provided it is non-empty). Also it is continuous on $\C^2$, hence the upper semi-continuous regularisation of $\Gr^+_{\seq{h}}$ on $\C^2$ matches with itself and \textit{Step 2} holds.

\smallskip\noindent A similar argument will work for $\Gr^-_{\seq{h}}$, which completes the proof.
\end{proof}
\begin{cor}
There exist constants $c_{\seq{h}}^\pm \in \mathbb{R}$ such that for $(x,y)\in V_R^\pm$ (respectively),
\begin{align*}
 	\Gr^+_{\seq{h}}(x,y)=\log|y|+O(1) \text{ and }
\Gr^-_{\seq{h}}(x,y)=\log|x|+O(1).
\end{align*}

\end{cor}
\begin{proof} The proof is same as the proof of Corollary \ref{c:Green constant}.
 \end{proof}
\begin{lem}\label{l:NA uniform convergence}
The sequences $\{\Gr_k^\pm\}$ converges uniformly on compact subsets of $\C^2$ to $\Gr_{\seq{h}}^\pm$, respectively.
 \end{lem}
 \begin{proof}
The proof is again completely similar to the proof of Lemma \ref{l:uniform convergence}, however, we revisit the steps briefly. Note that if $C \subset \U_{\seq{h}}^+$ then the uniform convergence is immediate, as  the sequence $\seq{\Gr^+}$ is uniformly Cauchy on $\U_{\seq{h}}^+$, by the proof of Theorem \ref{t: NA Green function}.

\smallskip  Next, let $C \subset \K_{\seq{h}}^+$ then $h(k) \in \G^b_k(z)$ for every $z \in C$ and $k\ge 1$, where $\G^b_k(z)$ is as introduced in the \textit{Step 3} of the proof of Theorem \ref{t:result 1}, and hence by Remark \ref{r:step 4}, there exists a positive integer $\ku_0(C)\ge 1$ such that $h(k)(C)\subset V_R$ for every $k \ge \ku_0$. Thus ${\Gr_{k}^+}_{|C}\le \frac{R}{2^k}$, which proves the uniform convergence in this case.

\smallskip Finally, let $C$ intersects both $\K_{\seq{h}}^+$ and $\U_{\seq{h}}^+$, then the uniform convergence is immediate from above on $C \cap \K^+_{\seq{h}}$, i.e., for a given $\epsilon>0$, there exists $\ku_0 \ge 1$ such that $|\Gr_k^+(z)-\Gr_{\seq{h}}^+(z)| \le \epsilon$ for every $k \ge \ku_0.$ Further by Lemma \ref{l:step 4} there exists $\ku_0(C) \ge 1$ such that $h(k)(C) \subset V_R \cup V_R^+$ for every $k \ge \ku_0(C)$. Note that  by (\ref{e:VR+}), 
\begin{align}\label{e:NA inclusion}
	h(k)^{-1}(V_{R}^+)=\ov{h(k)^{-1}(V_{R}^+)} \subset \I{h(k+1)^{-1}(V_{R}^+)}.
\end{align}
Now as in the proof of Lemma \ref{l:uniform convergence} we define the following subsets of $C$ 
	\[ C_k=C \cap\big( h(k)^{-1}(\I{V_{R}^+}) \setminus  h(k-1)^{-1}(\I{V_{R}^+})\big) \text{ and } C_0=C \cap \I{V_{R}^+},\]
i.e., $\cup_{k=0}^\infty C_k=C \cap \U_{\seq{h}}^+.$ Since $C \cap \U_{\seq{h}}^+\neq \emptyset$, it follows from (\ref{e:NA inclusion}) that $C_k$'s are non-empty sets for $k \ge 1$, sufficiently large. Also, let $B=\max\{\|\h_i(z)\|: z \in V_{R+1}, 1 \le i \le n_0\}$ and $\tilde{C}_k=\cup_{i=k}^\infty C_i.$ Then for $z \in C_k$, $\Gr_l^+(z) \le \frac{\log B}{\du_l}$ whenever $\ku_0(C) \le l \le k.$ Now for $n \ge 1$
	\[ \Gr_{l+n}^+(z) \le \frac{\log B}{\du_k}+ \log{M} \sum_{i=1}^{n}\frac{1}{\du_{k+i}} \le \frac{2\widetilde{M}}{\du_k} \le \frac{\widetilde{M}}{2^{k-1}},\]
where $\widetilde{M}=\max\{|\log M|, |\log m|,|\log B|\}.$ Again, by the same arguments as in proof of Lemma \ref{l:uniform convergence}, i.e., by continuity of $\Gr^+_{\seq{h}}$ and the above, there exists $\ku_1\ge 1$ such that
	\[ |\Gr_k^+(z)-\Gr_{\seq{h}}^+(z)| < \ep\]
whenever $z \in \widetilde{C}_{\ku_1}$ and $k \ge \ku_1.$
Now, as 
\[(\U^+_{\seq{h}} \cap C)\setminus \widetilde{C}_{\ku_1} \subset C \cap h(\ku_1-1)^{-1}(V_R^+),\] and $C \cap h(\ku_1-1)^{-1}(V_R^+)$ is a compact set contained in $\U_{\seq{h}}^+$, there exists $\ku_2 \ge 0$ such that $|\Gr_k^+(z)-\Gr_{\seq{h}}^+(z)| \le \ep$ whenever $z \in (\U^+_{\seq{h}} \cap C)\setminus \widetilde{C}_{\ku_1}.$ Since $C=(C\cap \K^+_{\seq{h}}) \cup \widetilde{C}_{\ku_1} \cup ((\U^+_{\seq{h}} \cap C)\setminus \widetilde{C}_{\ku_1})$, for $k \ge\max\{\ku_0,\ku_1,\ku_2\}$, we have $\big|\Gr_k^+-\Gr_{\seq{h}}^+\big|_{C} \le \ep.$
\end{proof}
\begin{thm}\label{t: NA sequential convergence}
	For every $k \ge1$, let $G^\pm_{h(k)}$ denote the Green's function corresponding to the maps $h(k)$ and $h^{-1}(k).$ Then the sequence $\{G^\pm_{h(k)}\}$ converge uniformly to $\Gr_{\seq{h}}^\pm$, respectively, on compact subsets of $\C^2.$
 \end{thm}
\begin{proof}
This proof is again similar to the proof of Theorem \ref{t:convergence of mu}. Let $C$ be a compact subset of $\C^2$, then let $C_1=\K_{\seq{h}}^+ \cap C$. Since $C_1$ is a compact set contained in $\K_{\seq{h}}^+$, by Lemma \ref{l:step 4} and Remark \ref{r:step 4} there exists $ \ku_{C_1} \ge 1$ such that $h(k)(z) \in V_R$ for $k \ge \ku_{C_1}$. Thus $h(k)\in \G^b_k(z)$ for $z \in C_1$ and $h(k)\in \G^u_k(\tilde{z})$ for $\tilde{z} \in C \setminus C_1$, whenever $k \ge \ku_{C_1}$. 

\smallskip
Now by \textit{Step 1} in the proof of Theorem \ref{t:convergence of mu} gives that, for a given $\ep>0$ there exists $\ku_{1} \ge \ku_{C_1}$ such that for $k \ge \ku_1$, 
\[\bigg|G^+_{h(k)}(z)-\frac{\log^+\norm{h(k)(z)}}{\du_k}\bigg|_{C_1}=|G^+_{h(k)}(z)-\Gr^+_k(z)|_{C_1}<\ep/2.\]
Also as $h(k) \in \G^u_k(\tilde{z})$ for $\tilde{z} \in C \setminus C_1$ and $k \ge \ku_1$, by \textit{Step 2} in the proof of Theorem \ref{t:convergence of mu}   
	\[\frac{\log\|h(k)(\tilde{z})\|}{\du_k} -\widetilde{M}\sum_{i=1}^{l-1}  \frac{1}{\du_k^i} \le \frac{\log\|h(k)^l(\tilde{z})\|}{\du_k^l} \le \frac{\log\|h(k)(\tilde{z})\|}{\du_k} +\widetilde{M}\sum_{i=1}^{l-1}  \frac{1}{\du_k^i},\]
	whenever $l \ge 2$. Hence there exists $\ku_2 \ge \ku_1$ such that for $k \ge \ku_2$
	\[\bigg|G^+_{h(k)}(z)-\frac{\log^+\norm{h(k)(z)}}{\du_k}\bigg|_{C\setminus C_1}=|G^+_{h(k)}(z)-\Gr^+_k(z)|_{C\setminus C_1}\le \frac{2\widetilde{M}}{\du_k}<\ep/2.\]

\smallskip
Finally, by Lemma \ref{l:NA uniform convergence}, there exists $\ku_C \ge \ku_2$ such that the theorem holds. A similar argument will work for $\{G_{h(k)}^-\}$ and $\Gr^-_{\seq{h}}.$
\end{proof}
Now $\Gr^\pm_{\seq{h}} \equiv 0$ in the $\I{\K_{\seq{h}}^\pm}$, provided it is non-empty, hence $\K_{\seq{h}}^\pm$ are pseudoconcave subsets of $\C^2$. Also, as an immediate corollary to Theorems \ref{t: NA Green function} and \ref{t: NA sequential convergence}, we have the initial statement of the following. 
 \begin{cor}\label{c:NA support}
The currents $\displaystyle \mu_{\seq{h}}^\pm:=\frac{1}{2\pi} dd^c(\Gr^\pm_{\seq{h}})$ are positive $(1,1)$ currents of mass 1, supported on $\J_{\seq{h}}^\pm=\partial \K_{\seq{h}}^\pm$, respectively. Also support of $ \mu_{\seq{h}}^\pm=\J_{\seq{h}}^\pm$ and $\mu_{\seq{h}}:=\mu_{\seq{h}}^+ \wedge \mu_{\seq{h}}^-$ is a compactly supported probability measure.
\end{cor}
\begin{proof}
We only prove $\text{Supp } \mu_{\seq{h}}^\pm=\J_{\seq{h}}^\pm$, here. Since $\Gr_{\seq{h}}^+$ is non-constant on $\C^2$ and attains the minimum value, i.e., zero, in the interior of neighborhood of a point $z_0 \in \J_{\seq{h}}^+$, the function is strictly pluriharmonic at $z_0.$ As $z_0$ is an arbitrary point on $\J_{\seq{h}}^+,$ the support of $\mu_{\seq{h}}^+$ is equal to $\J_{\seq{h}}^+$. A similar argument will work for $\mu_{\seq{h}}^-.$ Also $\mu_{\seq{h}}$ is a positive measure is immediate, and it is compactly supported follows from Remark \ref{r:NA K+ closed}.
\end{proof}
\begin{rem}\label{r:NA Julia}
Note that any subsequence of $\{h(k)\}$ neither diverges to infinity nor is it bounded on any neighbourhood of a point $z_0 \in \partial \K_{\seq{h}}^+$. Thus $\J_{\seq{h}}^+=\partial \K_{\seq{h}}^+$ is contained in the Julia set for the dynamics of the non-autonomous family $\seq{h}$. But note that \[\C^2 \setminus \J_{\seq{h}}^+=\I{\K_{\seq{h}}^+} \cup \U_{\seq{h}}^+.\]
Hence by Lemma \ref{l:step 4} and Remark \ref{r:step 4}, $\I{\K_{\seq{h}}^+}$ is contained in the Fatou set and thus the Julia set corresponding to the dynamics of $\seq{h}$ is equal to $\J_{\seq{h}}^+$.
\end{rem}
\begin{rem}\label{r:conditions NA}
Note that the two `crucial' conditions required on a non-autonomous sequence of H\'{e}non maps $\seq{h}$ of the form (\ref{e:fghm}), to complete the proof of Theorem \ref{t: NA Green function} and \ref{t: NA sequential convergence} are
\begin{enumerate}[leftmargin=14pt]
\item[(i)] The sequence $\seq{h}$ admits a uniform radius filtration  $R_{\seq{h}}>1$ (in the above case it is the radius of filtration of the semigroup $\s$, generated by $\G$), such that for every $R>R_{\seq{h}}$
\begin{itemize}
	\item $\displaystyle \overline{h_k(V_R^+)} \subset V_R^+ \text{ and } \overline{h^{-1}_k(V_R^-}) \subset V_R^-.$
	\item there exists a sequence positive real numbers $\seq{R}$ diverging to infinity, with $R_0=R$, satisfying
	$\displaystyle V_{R_k} \cap h(k)(V_R^+)=\emptyset \text{ and }V_{R_k} \cap h^{-1}(k)(V_R^-)=\emptyset.$
	\item There exist uniform constants $0<m<1<M$, such that the filtration identities (\ref{e: NA filtration+}) and (\ref{e:NA filtration-}) are satisfied on $V_R^\pm$, respectively.
\end{itemize}
\item[(ii)] For every $R\ge R_{\seq{h}}$, there exists a uniform constant $B_R=\max\{\norm{h_k(z)}: z\in V_R\} < \infty.$ The same holds in the above setup of Theorem \ref{t: NA Green function} and \ref{t: NA sequential convergence}, as the choices for $h_k$ are finite, for every $k \ge 1.$ 
\end{enumerate} 
\end{rem}
\noindent Hence we have the following analogue of of Theorem \ref{t: NA Green function} and \ref{t: NA sequential convergence} in a more general setup.
\begin{rem}\label{r:general NA}
Let $\seq{h}$ be a non-autonomous sequence of H\'{e}non maps satisfying conditions (i) and (ii) of Remark \ref{r:conditions NA} then 
\begin{itemize}[leftmargin=14pt]
\item The sequences of plurisubharmonic function $\seq{\Gr^\pm}$, as defined in (\ref{e:NA Green sequence}) converges to a plurisubharmonic continuous functions $\Gr_{\seq{h}}^\pm$ on $\C^2$, respectively. Further $\Gr^\pm_{\seq{h}}$ is pluriharmonic on $\U_{\seq{h}}^\pm$ and $\I{\K_{\seq{h}}^\pm}$, where  $\U_{\seq{h}}^\pm$ and $\I{\K_{\seq{h}}^\pm}$.
	
\smallskip
\item The sequences $\{G^\pm_{h(k)}\}$ converge uniformly to $\Gr_{\seq{h}}^\pm$, respectively, on compact subsets.
\end{itemize}
\end{rem}
\begin{ex}\label{ex}
Let $H_k(x,y)=(a_k y, a_k x+p(y))$ where $p$ is a polynomial of degree at least $2$, then the sequence $\{\tilde{H}_k\}$ defined as below is a sequence of H\'{e}non maps.
\begin{align*}
\tilde{H}_k(x,y):=H_{2k} \circ H_{2k-1}(x,y)&
&=(y,a_{2k}a_{2k-1}x+p(y/a_{2k})) \circ (y,a_{2k}a_{2k-1}x+a_{2{k}}p(y)).
 \end{align*}
  Further, if $0<c<|a_k|<d$ for every $k \ge 1$, the conditions (i) and (ii) in Remark \ref{r:conditions NA} are satisfied, and by Remark \ref{r:general NA}, it is possible to construct the dynamical Green's functions. However, the condition (i) in Remark \ref{r:conditions NA} fails for $\tilde{H_k}^{-1}$, if $|a_k| \to 0$ (see Theorem 1.4 in \cite{F:short}).
  \end{ex}
Also note that the functions $\gr$ admit logarithmic growth at infinity, and the closure of the sets $\K_{\seq{h}}^\pm$ in $\mathbb{P}^2$ is $\K_{\seq{h}}^\pm \cup I^\pm$, as defined in Section \ref{s:5}. Hence it is possible to generalise the results stated in Section \ref{s:5}, to the setup of dynamics of a non-autonomous sequence of H\'{e}non maps $\seq{h}\subset \s$. In particular, the analogue to Corollary \ref{t:result 4} is
\begin{cor}\label{t:NA currents}
Let $S^\pm$ be two closed positive $(1,1)$-currents in $\mathbb{P}^2$ of mass 1, such that the support of  $S^+$ does not contain the point $[0:1:0]$ and the support of  $S^-$ does not contain the point $[1:0:0]$. Also, let $\overline{h(k)}$ denote the extension of $h(k)$ to $\mathbb{P}^2,$ for every $k \ge 1$ then 
\[\lim_{k \to \infty}\frac{1}{\du_k}  \overline{h(k)}^*(S^+) \to \mu_{\seq{h}}^+ \text{ and }\lim_{k \to \infty}\frac{1}{\du_k}  \overline{h^{-1}(k)}^*(S^-) \to \mu_{\seq{h}}^-.\]
\end{cor}
The  proof is immediate from Remark \ref{r:NADS} and Theorem \ref{t: NA sequential convergence}. Also the proof of Corollary \ref{t:NA currents} does not generalises to general non-autonomous families of H\'{e}non maps (observed in Remark \ref{r:conditions NA}), unlike Theorems \ref{t: NA Green function} and \ref{t: NA sequential convergence}. It crucially requires that $h(k) \in \G_k$, $k \ge 1.$
\begin{rem}\label{r:analogy}
Note that as mentioned in the introduction, the above result is a more explicit version of Theorem 5.1 in \cite{DS:horizontal}, for H\'{e}non maps. The latter established the existence of similar non-autonomous currents for families of horizontal maps on appropriate subdomains of $\C^k$, $k \ge 2$ and H\'{e}non maps of the above form are indeed known to be horizontal on a large enough polydisc at the origin in $\C^2$, by \cite{S:rationalle}. Also, the construction and the convergence properties of similar Green's current for parametrised families of skew-product of (monic) H\'{e}non maps\,---\,of fixed degree\,---\,over compact complex manifolds, have been studied in \cite{PV: paper 1}.
\end{rem}
\section{Attracting basins of non-autonomous sequences in $\s$ }\label{s:7}
Let $\s$ be a semigroup generated by finitely many H\'enon maps, having an attracting behaviour, i.e., satisfy (\ref{e:uniformly attracting}) at the origin. Then for every $i \ge 1$, there exist $r>0$ and $0< \alpha < 1 $ such that $\h_i(B(0;r)) \subset B(0;\alpha r)$. In particular, for every sequence $\{h_k\} \subset \s$, $h_k(z) \to 0$ as $n \to \infty$ for $z \in B(0;r)$. Hence we have the following observations.
\begin{itemize}[leftmargin=14pt]
\item The strong filled positive Julia set $\ko_\s^+$ is non-empty and contains a neighbourhood of the origin. Also, the strong filled negative Julia set $\ko_\s^-$ is non-empty and contains the origin.
	
\smallskip
\item The basin of attraction at the origin of every $h \in \s$, say $\Omega_h$, is a Fatou--Bieberbach domain, i.e., biholomorphic to $\C^2$ (see \cite{RRpaper} for the proof).

\smallskip
\item The non-autonomous basin of attraction at the origin for a sequence $\{h_k\}$\,---\,denoted by $\Omega_{\seq{h}}$, defined in statement of Theorem \ref{t:BC}\,---\,is an \emph{elliptic} domain containing the origin as every $h_k$ satisfies the uniform bound condition. (see \cite{FSpaper},\cite{FW:Proceedings} for the result).
 
 \end{itemize}
 \begin{lem}\label{l:boundary}
  $\partial \Omega_{\seq{h}}\subset \partial \K_{\seq{h}}^+$ for the non-autonomous dynamical system $\seq{h}.$
 \end{lem}
\begin{proof}
Observe that by the argument as in Remark \ref{r:NA Julia}, $\partial \Omega_{\seq{h}}\cap \U_{\seq{h}}^+=\emptyset$. Now, if $z_0 \in \partial \Omega_{\seq{h}} \cap \I{\K_{\seq{h}}^+}$, then there exists a neighbourhood of $B(z_0;\delta) \subset \I{\K_{\seq{h}}^+}$, i.e., the sequence $\{h(k)\}$ is locally uniformly bounded and hence normal on $B(z_0;\delta)$. Then there exists a subsequence $\{n_k\}$ such that $h(n_k)(z_0)$ does not tend to $0$, however $h(n_k)(z_1)\to 0$  whenever $z_1\in B(z_0;\delta) \cap \Omega_{\seq{h}}^+$. Thus $\{h(k)\}$ is not normal on any neighbourhood of $z_0$. Hence by Remark \ref{r:NA Julia}, $\partial \Omega_{\seq{h}} \subset \J_{\seq{h}}^+= \partial \K_{\seq{h}}^+$.  
\end{proof}
\begin{rem}
Note that in the setup of iterative dynamics of H\'{e}non maps, the boundary of \emph{any} attracting basin is equal to the Julia set (see \cite[Theorem 2]{BS0}). In the above lemma we only show that $\partial \Omega_{\seq{h}}$ is properly contained in the Julia set. This leads to the question: Is $\J_{\seq{h}}^+ = \partial \Omega_{\seq{h}}$ in the non-autonomous setup?
\end{rem}

 Let $\Omega_{h(k)}$ be the basin of attraction of the origin for every $h(k)$, $k \ge 1$, as origin is an attracting fixed point, i.e., the definitions are compared as
 \begin{align*}
\Omega_{h(k)}=\{z \in \C^2: h(k)^n (z) \to 0 \text{ as } n \to \infty\} \text{ and }
\Omega_{\seq{h}}=\{z \in \C^2: h(k) (z) \to 0 \text{ as } k \to \infty\}.
\end{align*} 
 \begin{lem}\label{l:compacts in basin}
Let $K$ be a compact set contained in $\Omega_{\seq{h}}$	 then there exists a positive integer dependent on $K$, i.e., $N_0(K) \ge 1$ such that $K \subset \Omega_{h(k)}$ for every $k \ge N_0(K).$
\end{lem}
\begin{proof}
 	Note that $\{h_k\}$ varies within  a collection of finitely many H\'{e}non maps, $\{\h_i: 1\le i \le n_0\}$, each admitting an attracting fixed point at the origin. Thus there exists a neighbourhood $B(0;r)$ at the origin and $0<\alpha<1$ such that $\h_i(B(0;r)) \subset B(0;\alpha r)$ for every $1 \le i \le n_0.$ In particular, $B(0;r)$ is contained in attracting basin of the origin for every $h$, $h \in \s$. Since $K \subset \Omega_{\seq{h}}$ is compact, $h(k)(w) \in B(0;r)$ for every $w \in K$ and $k \ge N_0(K).$ Hence  	\[G^+_{h(k)}(h(k)(w))=0, \text{ i.e., }G_{h(k)}^+(w)=0\]
	for every $w \in K$ and $n \ge N_0(K).$ So $K \subset \I{K^+_{h(k)}}.$ But $h(k)(0)=0$, hence $h(k)(\Omega_{h(k)})=\Omega_{h(k)}$, and the above implies $h(k)(K) \subset \Omega_{h(k)}$. Thus $K \subset \Omega_{h(k)}$ for every $k \ge N_0(K).$
\end{proof}
 Next, we complete the proof of Theorem \ref{t:BC}, by appealing to an idea used in \cite{W:FB domains}.
 \begin{proof}[Proof of Theorem \ref{t:BC}] Let $\seq{K}$ be an exhaustion by compacts of $\Omega_{\seq{h}}.$ Then from Lemma \ref{l:compacts in basin}, there exists an increasing sequence of positive integers $\seq{n}$ such that $K_k \subset \Omega_{h(n_k)}.$ Since  every $\Omega_{h(n_k)}$ is a Fatou-Bieberbach domain, our goal is to construct a sequence of biholomorphisms $\seq{\phi}$, i.e., holomorphic maps that are both one-one and onto from $\Omega_{h(n_k)}$ to $\C^2$, appropriately and inductively, such that for a given summable sequence of positive real numbers $\{\rho_k\}$ the following holds
 	\begin{align}\label{e:main}
 		 \norm{\phi_k(z)-\phi_{k+1}(z)}< \rho_k \text{ for } z \in K_k \text{ and }
 		\norm{\phi_k^{-1}(z)-\phi_{k+1}^{-1}(z)}< \rho_k \text{ for } z \in B(0;k).
 	\end{align}
 	\textit{Basic step: }Let $\phi_1: \Omega_{h(n_1)} \to \C^2$ be a biholomorphism. By results in \cite[Theorem 2.1]{AL:inventiones} for $\delta< \rho_1/2$ there exists $F_2 \in \Aut$ such that 
 	\begin{align}\label{e:B1}
 	 \norm{\phi_1(z)-F_{2}(z)}&< \delta \text{ for } z \in K_1(r) \text{ and } 
 		\norm{\phi_1^{-1}(z)-F_{2}^{-1}(z)}&< \delta \text{ for } z \in B(0;1+r)
 	\end{align}
 	where $K_1(r)=\cup_{z \in K_1} (B(z;r)) \subset \Omega_{h(n_1)}$ for some $r>0$, i.e., an $r$-neighbourhood of $K_1$, contained in $\Omega_{h(n_1)}$. Since $F_2$ is uniformly continuous on $K_1(r)$, there exists $\epsilon_0>0$ such that for $z,w \in K_1(r)$
 	\begin{align}\label{e:B2}
 		\norm{F_2(z)-F_2(w)}< \delta \text{ whenever }\norm{z-w}< \epsilon_0.
 	\end{align}
 	Let $\epsilon< \min\{\epsilon_0,r, \delta\}.$ Then from \cite[Lemma 4]{W:FB domains}, there exists a biholomorphism $\psi_2: \Omega_{h(n_2)} \to \C^2$ such that
 	\begin{align}\label{e:B3}
  \norm{\psi_2(z)-z}< \epsilon \text{ for every } z \in K_1 \text{ and }
 		\norm{\psi_2^{-1}(z)-z}< \epsilon \text{ for every } z \in F_2^{-1}(B(0;1)).
 	\end{align}
 	Thus for $z \in K_1$, $\psi_2(z) \in K_1(r)$ and by (\ref{e:B2}), (\ref{e:B3}) it follows that $\norm{F_2 \circ \psi_2(z)-F_2(z)}< \delta.$ Hence from (\ref{e:B1}), 
 	\[\norm{F_2\circ \psi_2(z)-\phi_1(z)}<2\delta< \rho_1 \text{ for }z \in K_1.\]
 	Also by (\ref{e:B3}), for $z \in B(0;1)$, $\norm{\psi_2^{-1} \circ F_2^{-1}-F_2^{-1}(z)}< \epsilon < \delta$. Again by (\ref{e:B1}), 
 	\[\norm{\psi_2^{-1}\circ F_2^{-1}(z)-\phi_1^{-1}(z)}<2\delta< \rho_1 \text{ for }z \in B(0;1).\]
 	Thus $\phi_1$ and $\phi_2:=F_2 \circ \psi_2$ satisfies (\ref{e:main}) for $k=1.$
 	
 	\smallskip\noindent
 	\textit{Induction step: }Suppose for $N \ge 2$, and there exist biholomorphisms $\phi_k: \Omega_{h(n_k)} \to \C^2$ such that (\ref{e:main}) is satisfied for every $1 \le k \le N-1.$ Our goal is to construct $\phi_{N+1}$ such that (\ref{e:main}) holds for $k=N.$ As before, for $\delta< \rho_{N}/2$, there exists $F_{N+1} \in \Aut$ such that 
 	\begin{align*}
 		\norm{\phi_{N}(z)-F_{N+1}(z)}< \delta \text{ for } z \in K_{N}(r) \text{ and } 
 		\norm{\phi_{N}^{-1}(z)-F_{N+1}^{-1}(z)}&< \delta \text{ for } z \in B(0;N+r),
 	\end{align*}
 	where $K_{N}(r)$ is an $r$-neighbourhood of $K_{N}$, contained in $\Omega_{h(k_N)}$  for some $r>0$. Since $F_{N+1}$ is uniformly continuous on $K_{N}(r)$, there exists $\epsilon_0>0$ such that for $z,w \in K_N(r)$
 	\begin{align}\label{e:G2}
 		\norm{F_{N+1}(z)-F_{N+1}(w)}< \delta \text{ whenever }\norm{z-w}< \epsilon_0.
 	\end{align}
 	Let $\epsilon< \min\{\epsilon_0,r, \delta\}.$ Then again by \cite[Lemma 4]{W:FB domains}, there exists a biholomorphism $\psi_{N+1}: \Omega_{h(n_{N+1})} \to \C^2$ such that
 	\begin{align}\label{e:G3}
	 \norm{\psi_{N+1}(z)-z}< \epsilon \text{ for } z \in K_{N} \text{ and }
	\norm{\psi_{N+1}^{-1}(z)-z}< \epsilon \text{ for } z \in F_{N+1}^{-1}(B(0;N)). 	
	\end{align}
 	Thus for $z \in K_N$, $\psi_{N+1}(z) \in K_N(r)$, and by (\ref{e:G2}), (\ref{e:G3}) it follows that $$\norm{F_{N+1} \circ \psi_{N+1}(z)-F_{N+1}(z)}< \delta.$$ Hence 
 	$\displaystyle\norm{F_{N+1}\circ \psi_{N+1}(z)-\phi_N(z)}<2\delta< \rho_N \text{ for }z \in K_N.$
 	Also, similarly as above, by (\ref{e:G3}), for $z \in B(0;N)$, $\displaystyle\norm{\psi_{N+1}^{-1} \circ F_{N+1}^{-1}(z)-F_{N+1}^{-1}(z)}< \epsilon < \delta$ and  by assumption on $F_{N+1}$, 
	\[\norm{\psi_{N+1}^{-1}\circ F_{N+1}^{-1}(z)-\phi_N^{-1}(z)}<2\delta< \rho_N.\]
	Thus $\phi_N$ and $\phi_{N+1}:=F_{N+1} \circ \psi_{N+1}$ satisfies (\ref{e:main}) for $k=N.$

 	\smallskip As $\seq{\rho}$ is summable, the sequences $\seq{\phi}$ and $\seq{\phi^{-1}}$ constructed converge on every compact subset of $\Omega_{\seq{h}}$ and $\C^2$, i.e., there exist analytic limit maps $\phi: \Omega_{\seq{h}} \to \C^2$ and $\tilde{\phi}: \C^2 \to \C^2.$ Since $\phi$ is a limit of biholomorphisms, either $\phi $ is one-one or $\text{Det }D\phi \equiv 0$ on $\Omega_{\seq{h}}.$ 

 	\smallskip Choose $A>0$ and $k \ge 1$, sufficiently large, such that $\sum_{i=k}^\infty \rho_i<A/2$.
	Also let $K=\phi_k^{-1}(B(0;A)$. Then $\vo{K}>0$, and by (\ref{e:main}), $B(0;A/2) \subset \phi(K) \subset B(0;3A/2)$, i.e., $\vo{\phi(K)}> \vo{B(0;A/2)}.$ But if $\text{Det }D\phi\equiv 0$, then $\vo{\phi(K)}=0$, which is not true. Hence $\phi$ is one-one on $\Omega_{\seq{h}}.$ 
	
 	\smallskip Finally, we prove that $\phi(\Omega_{\seq{h}})=\C^2.$ So first, observe that as a consequence of Theorem 5.2 in \cite{DE:Michael Problem}, $\phi_n^{-1}$ converges uniformly to $\tilde{\phi}$ on compact subsets of $\C^2$ and $\tilde{\phi}^{-1}=\phi$ on $\Omega_{\seq{h}}.$ Next we claim that for every positive integer $N_0 \ge 1$, $\tilde{\phi}(B(0,N_0)) \subset \Omega_{\seq{h}}.$ Suppose not, then there exists $z_0 \in \tilde{\phi}(B(0,N_0))$ such that $\Gr_{\seq{h}}^+(z_0)>0.$ Let $w_0=\tilde{\phi}^{-1}(z_0) \in B(0;N_0)$ and $z_k=\phi_k^{-1}(w_0).$ Then $z_k \in \Omega_{h(k)}$, $G_{h(k)}^+(z_k)=0$, for every $k \ge 1$ and $z_k \to z_0$ by (\ref{e:main}). But by Theorem \ref{t: NA sequential convergence}, $G_{h(k)}^+$ converges uniformly to $\Gr^+_{\seq{h}}$ on compact subsets of $\C^2$. Hence $\Gr^+_{\seq{h}}(z_0)=0$, which is a contradiction. Thus $\tilde{\phi}(B(0,N_0)) \subset \I{\K_{\seq{h}}^+}$ with $\tilde{\phi}(0) \in \Omega_{\seq{h}}$. Hence by Lemma \ref{l:boundary}, $\tilde\phi(\C^2) \subset \Omega_{\seq{h}}$ or $\C^2 \subset \phi(\Omega_{\seq{h}})=\C^2.$
 \end{proof}
 Also, the following is immediate from the above proof, and the Remarks \ref{r:conditions NA} and \ref{r:general NA}.
\begin{cor}\label{c:BC}
Let $\{\h_k\}$ be a sequence of H\'{e}non maps of form (\ref{e:fghm}), such that it satisfy 
\begin{itemize}[leftmargin=14pt]
\item	admits uniform filtration and bound conditions (i) and (ii) stated in Remark \ref{r:conditions NA}, and
\item is (upper) uniformly attracting  on a neighbourhood of origin, i.e., satisfying (\ref{e:uniformly attracting}).
\end{itemize}
  Then the basin of attraction of the sequence $\{\h_k\}$ at the origin is biholomorphic to $\C^2$.
\end{cor}
 Further, for parametrised families of H\'{e}non maps over compact manifolds, we have
\begin{ex}\label{e:skew}
	Let $M$ be a compact complex manifold and $\mathcal{H}: M \times \C^2 \to M \times \C^2$ be a skew product of H\'{e}non map parametrised over $M$, i.e., $\mathcal{H}(\lambda, x, y)=(\sigma (\lambda), \h_\lambda(x,y))$
	such that $\sigma$ is an (holomorphic) endomorphism of $M$ and $\h_\lambda$ is a H\'{e}non map of a \emph{fixed} degree $d \ge 2$ for every $\lambda \in M$, i.e., the family $\{\h_\lambda\}_{\lambda \in M}$ satisfies conditions (i) and (ii) of Remark \ref{r:conditions NA}. Further, if the family $\{H_\lambda\}_{\lambda \in M}$ is uniformly attracting on a neighbourhood of origin, i.e, it satisfies (\ref{e:uniformly attracting}), then for every $p=(p_1,0,0)$, the stable manifold $\Sigma^s_{\mathcal{H}}(p)$, defined as
	\[\Sigma^s_{\mathcal{H}}(p):=\{(\lambda,z) \in M \times \C^2: \mathcal{H}^k(\lambda, z) \to p \text{ as }k\to \infty\},\] is biholomorphic to $\C^2$, provided it is non-empty. This, in fact also answers a few particular cases of \textit{Problem 38} and \textit{39}, stated in \cite{Survey}.
\end{ex}
Finally, we conclude with an analytic property of the strong filled Julia set $\ko_\s^+.$ 
\begin{prop}\label{p:last}
Suppose there exists $1 \le i \neq j \le n_0$ such that $K_{\h_i}^+ \neq K_{\h_j}^+$, then $\Omega_{\seq{h}} \not\subset \ko_\s^+$, the strong filled Julia set, for every $\{h_k\} \subset \s.$
\end{prop}
\begin{proof}
We first claim that for every $h \in \s$, $\Omega_h \not\subset \ko_\s^+.$ Note that by definition $\ko_\s^+ \subset \K_h^+$ for every $h \in \s.$ If $\Omega_h \subset \ko_\s^+$ then $\partial\Omega_h=J_h^+ \subset \ko_\s^+.$ By assumption on $\s$, there exists (at least one) $g \in \s$ such that $K_g^+ \neq K_h^+.$ Now from the above $J_h^+ \subset \ko_\s^+ \subset K_g^+$, i.e., $\mu_h^+$ is a positive closed $(1,1)$ current supported on $K_g^+.$ Hence from Theorem 6.5 in \cite{Dinh-Sibony}, $\mu_h^+=\mu_g^+$ or $K_h^+=K_g^+$, which is a contradiction to the assumption.
	
\smallskip Now suppose there exists a sequence $\{h_k\} \subset \s$, such that $\Omega_{\seq{h}} \subset \ko_\s^+.$ Define the sequence $\tilde{h}_k=h(k)\circ h_1^{-1}$, for $k \ge 2$ Thus $z \in \Omega_{\seq{h}} \subset \ko_\s^+$, $\tilde{h}_k(z)$ is bounded, i.e., $h(k) \circ h_1^{-1}(z)$ is bounded. Let $F_k(z):=\|h(k) \circ h_1^{-1}(z)\|$. Then $F_k$ is a sequence of positive pluri-subharmonic functions on $\Omega_{\seq{h}}$. Further on any compact subset of $\Omega_{\seq{h}}$, all $F_k$'s, except finitely many is bounded uniformly by $R_\s$, where $R_\s$ is the radius of filtration of the semigroup $\s$ (as in Remark \ref{r:filtration estimate}.

\smallskip Next, let $F(z)=\limsup F_k(z)\text{ for } z \in \Omega_{\seq{h}}.$ Hence from Theorem 2.6.3 in \cite{KlimekBook}, the upper semicontinuous regularisation of $F$, denoted by $F^*$ of $F$ is a bounded pluri-subharmonic function on $\Omega_{\seq{h}}.$ Also as $F(z)=0$ on $B(0;r_\s)$, and the Lebesgue measure of the set $\{z \in \Omega_{\seq{h}}: F(z) \neq F^*(z)\}$ is zero, it follows that $F^*(z)=0$ almost everywhere on $B(0;r_\s).$ Since $\Omega_{\seq{h}}$ is a Fatou-Bieberbach domain by Theorem \ref{t:BC}, it cannot admit any non-constant bounded pluri-subharmonic function. Thus $F^* \equiv 0$ on $\Omega_{\seq{h}}$, i.e., $h(k)(w) \to 0$ for every $w \in h_1^{-1}(\Omega_{\seq{h}})$. Hence $h_1^{-1}(\Omega_{\seq{h}}) \subset \Omega_{\seq{h}}.$

\smallskip  Let $\mathcal{D} \subset B(0;r_\s) \subset \Omega_{\seq{h}}\subset \ko_\s^+ $ be a relatively compact subset of an one-dimensional algebraic variety such that $\partial{\mathcal{D}} \cap J_{h_1}^-=\emptyset$ and $[\mathcal{D}] \wedge \mu_{h_1}^-=c \neq 0.$ Thus by Corollary 1.7 of \cite{BS3}, \[S_n=\frac{1}{d_{h_1}^n} {h_1^n}^*([\mathcal{D}]) \to c \mu_{h_1}^+ \text{ as }n \to \infty.\]
Note $S_n$'s are positive $(1,1)$-currents supported on $h_1^{-n}(\mathcal{D}) \subset h_1^{-n}(\Omega_{\seq{h}}) \subset \Omega_{\seq{h}} \subset \ko_\s^+$ (from the previous observation). Hence $\mu_{h_1}^+$ is supported on $\ko_\s^+$, in particular $J_{h_1}^+ \subset \ko_\s^+$ which is not possible from the claim above.
\end{proof}
\begin{rem}
The above also proves that $\ko_\s^+$ cannot contain any $(1,1)$ positive closed current of finite mass and the positive Green's function $G_\G^+$ is unbounded on all\,---\,both autonomous and non-autonomous\,---\,basins of attraction.
\end{rem}
{\small  \bibliographystyle{amsplain}

}
\end{document}